\documentclass[11pt]{article}

\usepackage[margin=1in]{geometry}  
\usepackage{graphicx}           
\usepackage{amsmath, accents}              
\usepackage{amsfonts}             
\usepackage{mathrsfs}
\usepackage{amsthm}    
\usepackage{amssymb}           
\usepackage{soul}
\usepackage{xcolor}
\usepackage{float}
\usepackage{mathtools}
\usepackage{amsxtra}
\usepackage{amstext}
\usepackage{amsbsy,stmaryrd}
\usepackage{amscd}
\usepackage{siunitx}
\usepackage[colorlinks=true,breaklinks=true,linkcolor=lightgreen,citecolor=lightblue]{hyperref}

\definecolor{lightblue}{rgb}{0.22,0.45,0.70}
\definecolor{lightgreen}{rgb}{0.22,0.55,0.20}

\usepackage{todonotes}

\newtheorem{theorem}{Theorem}[section]
\newtheorem{lemma}[theorem]{Lemma}
\newtheorem{defi}[theorem]{Definition}

\theoremstyle{remark}

\newcommand{\mb}[1]{\mathbf{#1}} 
\newcommand{\Real}{\mathbb{R}}  
\newcommand{\Div}{\mathrm{Div}\,}  
\newcommand{\diver}{\nabla \cdot}

\newcommand{\bs}[1]{\boldsymbol{#1}}

\newcommand{\dx}{\ensuremath{\, \mathrm{d}x}}
\newcommand{\ds}{\ensuremath{\, \mathrm{d}s}}
\newcommand{\dz}{\ensuremath{\, \mathrm{d}z}}

\newcommand{\dt}{\ensuremath{\, \mathrm{d}t}}
\newcommand{\esssup}{\operatorname*{ess \, sup}}
\newcommand{\Dt}{\Delta t}
\newcommand{\Eh}{\mathcal{E}_h}
\newcommand{\Th}{\mathcal{T}_h}
\newcommand{\upw}{\mathrm{upw}}

\newcommand{\cblue}[1]{\textcolor{blue}{#1}}

\newcommand{\Thnorm}{1,\Th}

\DeclarePairedDelimiter\norm{\lVert}{\rVert}
\DeclarePairedDelimiter\snorm{\lvert}{\rvert}
\DeclarePairedDelimiter\abs{\lvert}{\rvert}

\newcommand{\Norm}{\norm{\cdot}}
\newcommand{\dmean}[1]{{\left\{\kern-0.6ex\left\{ #1 
		\right\}\kern-0.6ex\right\}}}
\DeclarePairedDelimiter\djump{\llbracket}{\rrbracket}

\begin{document}

\nocite{*}

\title{A Discontinuous Galerkin and Semismooth Newton Approach for the Numerical Solution of  Bingham Flow with Variable Density}

\author{Sergio~Gonz\'alez-Andrade \\ \footnotesize Research Center on Mathematical Modeling (MODEMAT) and \\\footnotesize Departamento de Matem\'atica -  Escuela Polit\'ecnica Nacional\\\footnotesize Ladr\'on de
Guevara E11-253, Quito 170413, Ecuador\\\footnotesize\tt sergio.gonzalez@epn.edu.ec\\ \\
Paul E. M\'endez Silva \\ \footnotesize Research Center on Mathematical Modeling (MODEMAT) -  Escuela Polit\'ecnica Nacional\\\footnotesize Ladr\'on de
Guevara E11-253, Quito 170413, Ecuador\\\footnotesize\tt paul.mendez01@epn.edu.ec} 
\date{\today}
\maketitle

\begin{abstract}
This paper is devoted to the study of Bingham flow with variable density. We propose a local bi-viscosity regularization of the stress tensor based on a Huber smoothing step. Next, our computational approach is based on a second-order, divergence-conforming discretization of the Huber regularized Bingham constitutive equations, coupled with a discontinuous Galerkin scheme for the mass density.  We take advantage of the properties of the divergence conforming and discontinuous Galerkin formulations to incorporate upwind discretizations to stabilize the formulation. The stability of the continuous problem and the full-discrete scheme are analyzed. Further,  a semismooth Newton method is proposed for solving the obtained fully-discretized system of equations at each time step. Finally,  several numerical examples that illustrate the main features of the problem and the properties of the numerical scheme are presented.
  \vskip .2in

\vspace{0.2cm}
\end{abstract}

\section{Introduction}
\subsection{Scope}
In this paper, we are interested in the analysis and numerical approximation of unsteady incompressible Bingham flow with variable density.  Let $\Omega\subset\mathbb{R}^d$ , $d = 2,3$,  be a bounded connected domain with Lipschitz-continuous boundary $\partial \Omega$ and let $T$ be a real positive number. Then, this kind of flows are governed by the following Navier-Stokes type system
\begin{equation} \label{eq:main_prob}
\begin{array}{ccc}
\partial_t \rho + \mb{u}\cdot \nabla\rho=0, & \text{in }\Omega \times ]0,T[,\vspace{0.2cm}\\
 \rho\partial_t (\mb{u}) + ( \rho\mb{u}\cdot\nabla)\mb{u}- \Div\bs{\tau}+\nabla p = \mb{f}, & \text{in }\Omega \times ]0,T[,\vspace{0.2cm}\\
\diver\mb{u}=0, & \text{in }\Omega \times ]0,T[,
\end{array}
\end{equation}
where,  the sought quantities are the density $\rho$, the velocity of the fluid $\mathbf{u}$ and the pressure $p$.  This system has been proposed as the classical model for non-homogeneous flow or flow with variable density of incompressible fluids (see \cite{Guermond2000,Guermond2009,Ladyz78,Liu2007}). In this paper,  we are concerned with Bingham flow. Therefore,  the fluid stress tensor $\bs{\tau}$ is given by
\begin{align*}
	 \begin{cases}
		\bs{\tau} = 2 \eta \mathbf{D}\mb{u} + \tau_s \frac{\mathbf{D}\mb{u}}{\abs{\mathbf{D}\mb{u}}} & \text{ if }\mathbf{D} \mb{u} \neq 0, \\
		\abs{\bs{\tau}} \leq \tau_s & \text{ if }\mathbf{D} \mb{u} = 0.
	\end{cases}
\end{align*}
Here, $\mathbf{D}\mathbf{u}$ stands for the symmetric part of the gradient, $\eta$ is the viscosity  and $\tau_s$ represents the yield stress.  Finally,  the system is endowed  with appropriate initial data
\begin{align*}
	\rho(0) = \rho_0, \mb{u}(0) = \mb{u}_0 \quad \text{in }\Omega \times \{0\},
\end{align*}
and boundary conditions in the following manner
\begin{align*}
	\rho(\mb{x},t) = \psi(\mb{x},t), \quad (\mb{x},t) \text{ in }\partial \Omega_{\mathrm{in}} \times ]0,T[,\\
	\mb{u}(\mb{x},t) = \mb{g}(\mb{x},t), \quad (\mb{x},t) \text{ in }\partial \Omega \times ]0,T[,
\end{align*}
where
\[\partial \Omega_{\mathrm{in}} = \{x \in \partial \Omega \,|\, \mb{g}(x)\cdot \mb{n}_{\partial \Omega} < 0\}\]
with $\mb{n}_{\partial \Omega}$ representing the outer unit normal vector at $\mathbf{x}\in \partial\Omega.$ Furthermore, we assume that $\int_{\partial \Omega}\mb{n}_{\partial \Omega} \cdot \mb{g} = 0$, $\diver \mb{u}_0=0$, and that the compatibility condition $\mb{n}_{\partial \Omega}\cdot \mb{g}(\mb{x},0)=\mb{n}_{\partial \Omega}\cdot\mb{u}_0|_{\partial \Omega}$ holds. 

Bingham is the seminal model for viscoplastic fluids, which are materials whose rheology is defined by the existence of a yield stress, $\tau_s$. This characteristic implies that the material hardens in regions where the stress does not exceed the yield stress. Meanwhile, in the regions where the stress overpasses $\tau_s$, the material flows as a viscous fluid with plastic behaviour.  Because of this mechanical property, one particularity of Bingham fluids is the presence of rigid moving parts in the interior of the flow. The size and location of these rigid zones depend on the yield stress, and can even block the flow for high values of $\tau_s$.  This so-called \textit{blocking property} makes the study of these materials of interest in various fields and applications. For instance, when related to the flow of biological fluids, such as blood or mucus, a blocking could be an indicator of health-compromising phenomena (see \cite{Chat}).  Another of the main fields of applications is geophysical flows. In fact, the analysis of lava and volcanic material flows is of particular interest. Further, the most interesting and challenging applications in this area involve non-homogeneous and variable density flows, for instance, in the analysis of landslides \cite{Hild, Ionescu}.

As mentioned previously, our interest lies in studying incompressible fluids with viscoplastic Bingham behavior. To satisfy mass conservation in such fluids, two conditions must be met: the mass density of each fluid particle must remain constant during motion, and the velocity field must satisfy the incompressibility constraint. However, our focus is also on flows with variable density or non-homogeneity. We consider this non-homogeneity condition in the sense proposed by, e.g. \cite{Simon90,Guermond2009}, where a non-homogeneous fluid is understood as two (possible more) incompressible fluids with different densities which mix. For a variable density flow model, we need to consider a coupled system between a Navier-Stokes equation and a first-order transport equation for density, as shown in \eqref{eq:main_prob}. This makes the problem challenging from the PDEs theory perspective (\cite{Ladyz78,Lions1996}).  For mathematical theory on the well-posedness of variable density or non-homogeneous Bingham flow, we refer to \cite{Basov,Bohm,Danchin2004}. In particular, \cite{Bohm}, analyzes a variational formulation for non-homogeneous Bingham flow using the classical variational inequality approach and proves the existence of weak solutions for \eqref{eq:main_prob}. Further, the author finds regularity conditions to obtain uniqueness of solutions. In this work, we mainly focus on the numerical simulation of this flow problem, considering that the theoretical results hold.

The main challenge in simulating yield stress fluids, such as Bingham fluids, is to correctly represent the unyielded (rigid) and yielded (non-rigid) regions in the material. From the mathematical perspective, this implies developing strategies to deal with the intrinsic discontinuity in the stress tensor $\bs{\tau}$. Our approach in this work is based on a local regularization of the stress tensor in the Bingham constitutive equations.  The regularization approach has a well-known computational advantage: regularized systems can be solved by fast converging numerical algorithms, usually based on generalized Newton methods (See \cite{Saramito}). On the other hand, performing a smoothing step on the stress tensor modifies the expected modeled behavior. In our case, we seek a balance between efficient and fast computational solutions and a regularization process that keeps the physics of the flow as exact as possible. We have seen in previous contributions that this balance can be achieved with a Huber-type regularization process (\cite{JC2012,Gon1,GonMen}). The main idea of this smoothing process is that in order to model the yielded regions, we can consider the actual form of the stress tensor, while for the approximation of the unyielded regions, we consider a smooth version of the tensor. The intrinsic quality of this regularization lies in the fact that the regions in which the stress is modified can be very small and easy to represent computationally, which guarantees a reliable physical approximation of the flow. 

For developing numerical approximations to the regularized problem, it seems natural to look at the techniques established for the solution of homogeneous density incompressible Navier-Stokes equations and try to exploit them as much as possible.  It is the purpose
of this paper to advance a second-order divergence-conforming discretization for this problem. Specifically, we introduce an $\mathbf{H}(\mathrm{div})$-conforming method based
on Brezzi-Douglas-Marini (BDM) spaces \cite{Brezzi1985}, coupled with a discontinuous Galerkin discretization for density.  Both equations are stabilized with upwind terms as in \cite{Cockburn2005, DiPietro2012} and combined with an implicit, second-order backward differentiation formula (BDF2) for time discretization. 

Among the advantages of exactly divergence-free methods, we can mention the following: First, they are pressure-robust, which means that it is possible to separate velocity and pressure completely in the error analysis. Also, using an $\mathbf{H}(\mathrm{div})$-conforming FEM allows the usage of discontinuous Galerkin Finite element method (dG-FEM) techniques in the formulation analysis and treatment of the convective term. Moreover, the requirement for less stability implies that the amount of numerical dissipation added is minimized. Finally, the conservation properties of the exact equations of mass, energy, and momentum are naturally transferred to the discrete solution \cite{Schroeder2018}.

\subsection{Related Work}
While there is a rich body of literature on the numerical approximation of the constant density and viscosity Navier-Stokes equations, fewer results are available for the variable density case. The numerical approximation of similar coupled flow systems has been studied using many different numerical methods, including projection methods \cite{Pyo2007,Guermond2000}, fractional-step methods \cite{Freignaud2001,Guermond2009}, and the discontinuous Galerkin (dG) method \cite{Liu2007}. Furthermore, the numerical simulation of the variable density incompressible Navier-Stokes system was studied in \cite{Calgaro2008}, where the authors introduce a hybrid scheme that combines a Finite Volume approach for treating the mass conservation equation and a Finite Element method to deal with the momentum equation and the divergence-free constraint.

The $\mathbf{H}(\mathrm{div})$-conforming approach for the Brinkman equation was numerically studied by \cite{konno2011}, while exactly divergence-free $\mathbf{H}(\mathrm{div})$-conforming finite element methods for time-dependent incompressible viscous flow problems have been extensively studied in \cite{Schroeder2018}, with special emphasis on pressure and Reynolds semi-robustness of the formulations. 

In the case of variable density or density-dependent Bingham flow,  mixing and interaction of materials with different densities are a mainly interesting field for engineering and mathematical communities.  For instance, several contributions have discussed this model as a suitable background for landslides and, in general, for debris flows (see \cite{Hild,Hungr,Ionescu}). This assertion arises from the fact that debris flows involve several substances, including mixtures and suspensions of granular particles in water, sand, and organic matter, among others. Depending on the physical and mechanical conditions, these substances can create rigid zones that move within the flow, leading to the expected behavior of a viscoplastic Bingham material. Further, the flow is not expected to be homogeneous, as the density varies depending on the concentration of the component substances. 

One interesting and challenging benchmark problem is the so-called Rayleigh-Taylor instability that occurs when two fluids with different densities interact. In \cite{Demian,Dolude}, the authors analyze this phenomenon for two viscoplastic materials using a volume of fluid (VOF) method and a hydrodynamic simulation based on the Bingham model. Additionally, in \cite{Bertola}, the authors perform an experimental study of the behavior of viscoplastic drops moving in a given medium, usually with different densities. In contrast to most of these contributions, this paper focuses on the computational simulation of these phenomena, based on the variational analysis of the constitutive PDEs for non-homogeneous Bingham flow. 

\subsection{Outline of the paper}
The remainder of this paper is organized as follows. In Section \ref{sec:cont-form}, we introduce the continuous formulation of problem \ref{eq:main_prob} and recall its main properties. We also propose and briefly analyze the local Huber regularization for the problem.  In Section \ref{sec:Discr}, we describe the time semi-discretization, and then the complete discrete scheme of this problem, briefly addressing stability properties. We also discuss the semismooth Newton linearization of each time step. Finally,  in Section \ref{sec:Num-examples}, we illustrate the properties of the problem and the scheme with numerical examples generated by the method introduced. We close the paper with some remarks and discussions given in Section \ref{sec:conclusions}.

\section{The continuous formulation} \label{sec:cont-form}

In this section, we introduce and analyze a transient formulation of the coupled problem.  We start by introducing some notation. We denote by $L^p(\Omega)$ and  $W^{r,p}(\Omega)$ the usual Lebesgue and Sobolev spaces  with respective norms $\smash{\Norm_{L^p(\Omega)}}$ and $\smash{\Norm_{W^{r,p}(\Omega)}}$. If $p=2$ we write $H^r(\Omega)$ and $\Norm_{r,\Omega}$ in place of $\smash{W^{r,p}(\Omega)}$ and $\smash{\Norm_{W^{r,p}(\Omega)}}$. By~$\mb{L}$ and~$\mathbb{L}$ we denote the corresponding vectorial and tensorial counterparts of the scalar functional space~$L$, respectively. Further, we denote by $(\cdot,\cdot)_{\Omega}$ the usual inner product in $L^2(\Omega)$. Moreover, for any vector field $\mb{v} = (v_i)_{i = 1 , d}$ we set the gradient, symmetric part of the gradient and divergence, as
\begin{equation*}
	\nabla \mb{v} := \left(\frac{\partial v_i}{\partial x_j} \right)_{i,j=1,d}, \, \mathbf{D} \mb{v} :=\frac{1}{2} \left(\nabla \mb{v} + (\nabla \mb{v})^T\right)\,\mbox{  and  }\, \diver \mb{v} := \sum_{j=1}^d \frac{\partial v_j}{\partial x_j},
\end{equation*}
respectively. In what follows,  we usually use the vector-valued Hilbert spaces
	\begin{align*}
	\mb{H}(\mathrm{div};\Omega) &\coloneqq \bigl\{\mb{w} \in \mb{L}^2(\Omega): \diver \mb{w} \in L^2(\Omega) \bigr\}, \\
	\mb{H}_0(\mathrm{div};\Omega) &\coloneqq \bigl\{\mb{w} \in \mb{H}(\mathrm{div};\Omega): \mb{w}\cdot\mb{n}_{\partial\Omega} = 0 \text{ on }\partial\Omega \bigr\}, \\
	\mb{H}_0(\mathrm{div}\!^0;\Omega) &\coloneqq \bigl\{\mb{w} \in \mb{H}_0(\mathrm{div};\Omega) : \diver \mb{w} = 0 \text{ in }\Omega \bigr\},
	\end{align*}
For a given tensor $\bs{T}$, we let $\Div \bs{T}$ be the divergence operator acting along the rows of $\bs{T}$. We denote by $L^s(0,T;W^{m,p}(\Omega))$ the Banach space of all $L^s$-integrable functions from $[0,T]$ into $W^{m,p}(\Omega)$, with norm 
\begin{align*} 
	\norm{v}_{L^s(0,T;W^{m,p}(\Omega))} = \begin{cases} \displaystyle
		\left(\int_{0}^{T} \norm{v(t)}^s_{W^{m,p}(\Omega)} \dt \right)^{1/s} & \text{if $1\leq s < \infty$,} \\
		\esssup_{t \in [0,T]} \norm{v(t)}_{W^{m,p}(\Omega)} & \text{if $s = \infty$.} 
	\end{cases}
\end{align*}

\subsection{Huber regularization}
The main characteristic of viscoplastic materials is the existence of a yield stress. These fluids exhibit non-Newtonian behavior depending on this parameter: if the total stress is below the yield stress, the fluid moves without continuous deformation, which means that the material is moving as a rigid solid. This behavior is also expected in the so-called stagnation regions, where the material is at rest. On the other hand, if the stress surpasses the yield stress, the fluid flows as a Newtonian fluid in the particular case of the Bingham model.  

The complex behavior of Bingham fluids is modeled by the following stress structure:
\begin{equation}\label{stress}
\left\{\begin{array}{lll}
\bs{\tau}=2\eta\mathbf{D}\mathbf{u} + \tau_s\frac{\mathbf{D}\mathbf{u}}{\abs{\mathbf{D}\mathbf{u}}},& \mbox{if $\mathbf{D}\mathbf{u}\neq 0$}\vspace{0.2cm}\\ \abs{\bs{\tau}} \leq \tau_s,& \mbox{if $\mathbf{D}\mathbf{u}=0$}.
\end{array}\right.
\end{equation}
Note that in the so-called yielded regions, \textit{i.e.,} regions where $\mathbf{D}\mathbf{u}\neq 0$, the stress is given as a sum of two terms: a viscous term associated with the viscosity $\eta$, and a plastic term associated with the yield stress $\tau_s$. Furthermore, in the unyielded regions where $\mathbf{D}\mathbf{u}=0$, we only know that the stress is bounded. This is the main issue regarding the mathematical modeling and numerical solution of these materials: in general, we do not have a priori knowledge of the localization of the yielded or unyielded regions in the flow. Because of this fact, we are dealing with an ill-posed problem.

One classical approach for the analysis and numerical solution of these materials is to regularize the stress tensor. In this work, we propose a local regularization based on a Huber smoothing step, which, when applied to \eqref{stress}, reads as follows:
\begin{equation}\label{stress_reg}
\begin{split}
\bs{\tau}_\gamma:= \mu(\abs{\mb{Du}}_{\gamma})\mb{Du} \\
\mu(t) := 2\eta + \tau_s\gamma \frac{1}{t},
\end{split}
\end{equation}
where $|\bs{\mb{A}}|_\gamma:= \max\{\tau_s, \gamma|\mb{A}|\}$.  Here $\gamma\gg 0$ is a given regularization parameter, such that $\gamma\rightarrow\infty$. This is a local regularization approach, which has proven to be efficient and reliable for the numerical solution of several viscoplastic flow problems (see \cite{JC2012, Gon1, GonMen}). 

Several smoothing steps have been proposed for the Bingham model. The best-known regularization procedures are the Papanastasiou and the Bercovier-Engleman, which are built by using smooth (at least twice differentiable) functions (see \cite{FrigaardN}). In contrast, the Huber regularization (bi-viscosity) is based on a piecewise linear function that recovers the real structure of the stress in the yielded regions while making the smoothed region around the unyielded regions as small as possible. In Figure \ref{fig:reg}, left, we show a graphical comparison, in a 1D scheme, of the stress tensor $\bs{\tau}$ vs. the deformation tensor $\mb{D}\mb{u}$ for the regularization steps mentioned before. In this picture, it is possible to appreciate the qualitative advantage of Huber regularization. The regularized stress $\bs{\tau}_\gamma$ is the actual material stress in the yielded regions, while the regularization of the unyielded regions is performed in small neighbourhoods around the real rigid zones. This behaviour allows us to obtain precise and reliable approximations of the actual stress, even when we approach the unyielded regions \textit{i.e.} for small values of the deformation stress. This advantage can be also appreciated in Figure  \ref{fig:reg}, right, where it is shown that the Huber regularization approaches to the real model very aggressively with moderate values for the smoothing parameter.  In contrast with this behaviour, the other smooth regularization procedures depend on smooth functions, which implies that the approximation of the real Bingham behaviour is not precise, specially in regions close to the unyielded regions ($\mathbf{D}\mathbf{u}\rightarrow 0$).

 \begin{figure}
	\begin{center}
		\includegraphics[height=0.4\textwidth]{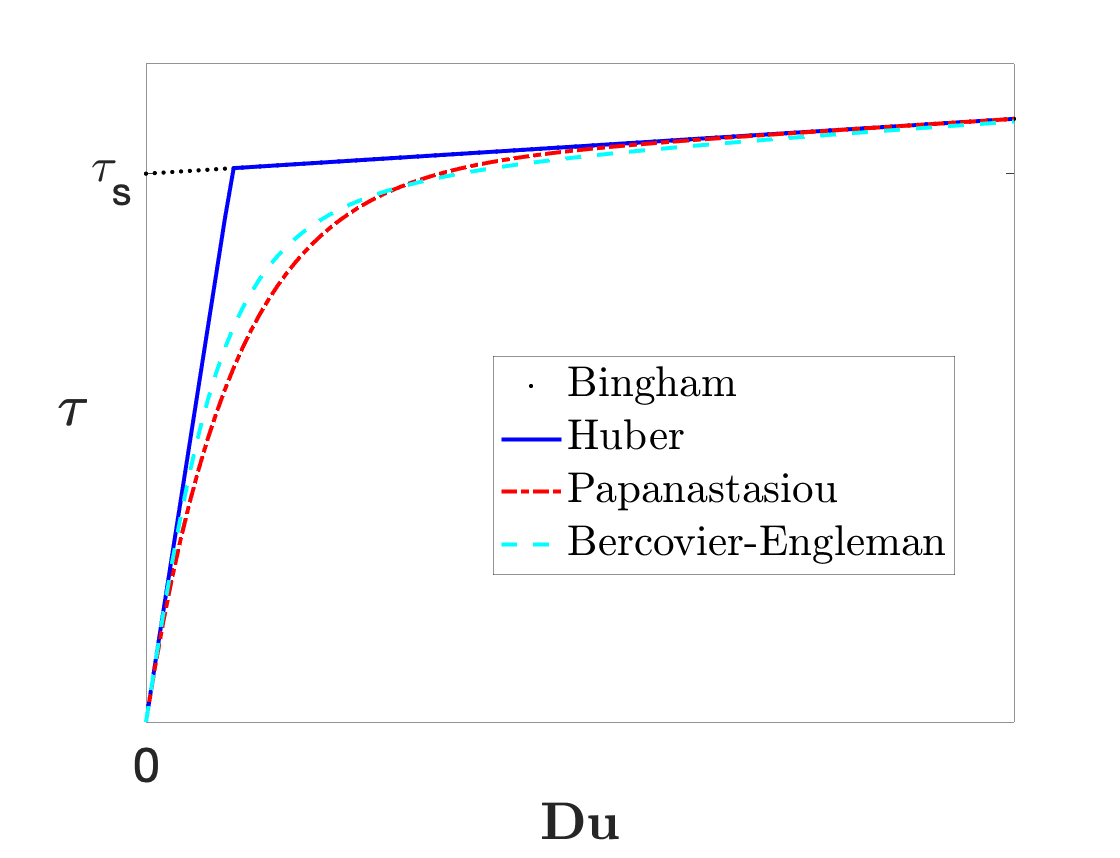}\includegraphics[height=0.4\textwidth]{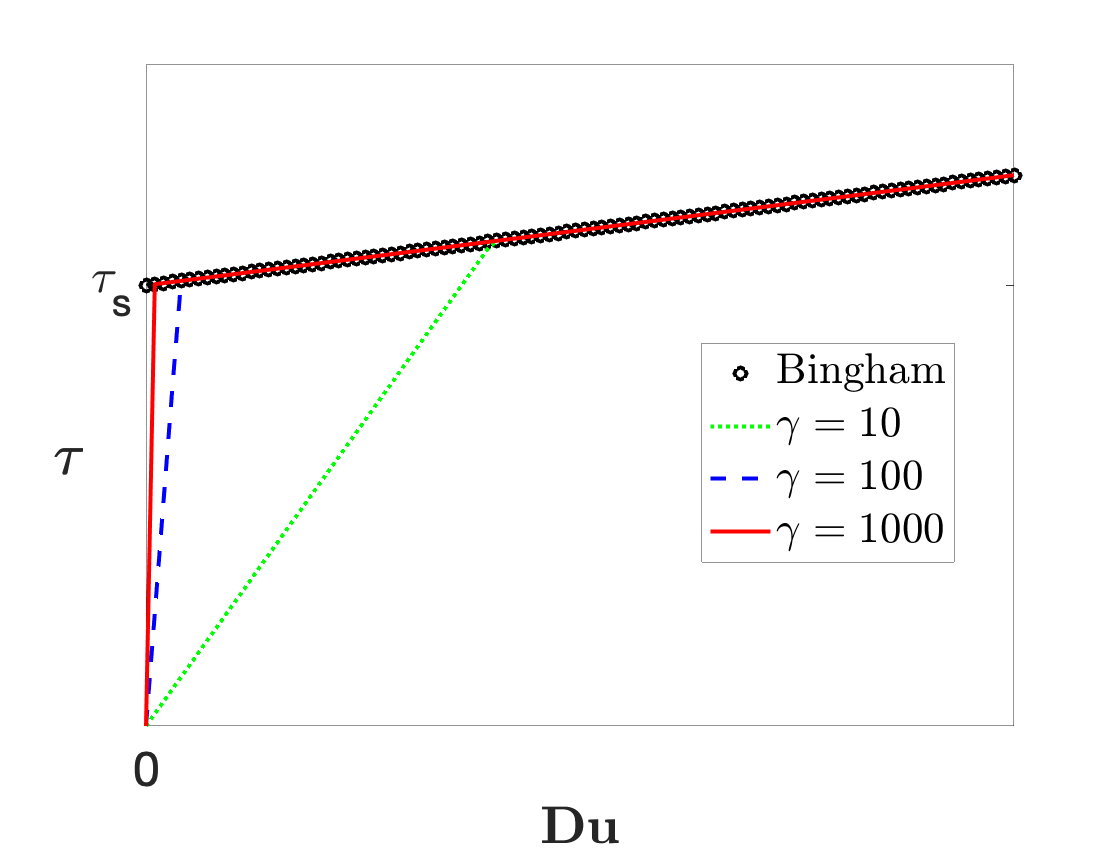}	
	\vspace{-3mm}
	\caption{Comparison between classical regularization schemes and the Huber (bi-viscosity) regularization (left). Huber regularization for several values of the smoothing parameter $\gamma$ (right).}\label{fig:reg}  
	\end{center}
\end{figure}

The Huber regularization is a local procedure designed to preserve qualitatively the structure of the model in the entire geometry.  Due to this fact, the smoothing approach allows us to directly define regions that approximate the yielded and unyielded regions in the flow in the following manner: the yielded regions are approximated by regions where $|\mb{Du}|\geq \frac{\tau_s}{\gamma}$, while the unyielded regions are approximated by regions where $|\mb{Du}|< \frac{\tau_s}{\gamma}$.  One of the main characteristics of the flow is the fact that the viscosity of the material is supposed to jump to infinity when crossing the separating phase from the yielded to the unyielded regions (let us recall that the model understands that the material moves like a rigid solid in the unyielded regions).  Considering that the parameter $\gamma\rightarrow\infty$,  the smoothing procedure sets a large viscosity in the unyielded regions and the actual viscosity of the material in the yielded regions (see \eqref{stress_reg}). Because of this fact, the approach is also known as bi-viscosity regularization (\cite{bever}). 

Summarizing, the system of Huber regularized constitutive equations for the non-homogeneous Bingham flow is given by
\begin{equation} \label{eq:reg_prob}
\begin{array}{ccc}
\partial_t \rho + \mb{u}\cdot \nabla\rho=0, & \text{in }\Omega \times ]0,T[,\vspace{0.2cm}\\
 \rho\partial_t ( \mb{u}) +  (\rho\mb{u}\cdot\nabla)\mb{u}- \Div\bs{\tau}_\gamma+\nabla p = \mb{f}, & \text{in }\Omega \times ]0,T[,\vspace{0.2cm}\\
\diver\mb{u}=0, & \text{in }\Omega \times ]0,T[,\vspace{0.2cm}\\
\bs{\tau}_\gamma:=\mu(\abs{\mb{Du}}_\gamma)\mb{Du},&\mbox{a.e. in $\Omega\times ]0,T[$}.
\end{array}
\end{equation}

\subsection{Weak formulation}

Let us define the following spaces
	\begin{align*}
		\mb{V} &\coloneqq \mb{H}^1_0(\Omega), \quad \mathcal{Q} \coloneqq L_0^2(\Omega), \quad \mathcal{W} \coloneqq H^1_0(\Omega),\\
		\mb{V}^t &\coloneqq \bigl\{\mb{w} \in \mb{L}^2(0,T;\mb{V}) \,:\, \partial_t \mb{w} \in L^2(0,T;\mb{L}^2(\Omega)) \bigr\}, \\
		\mathcal{Q}^t & \coloneqq L^2(0,T; \mathcal{Q}), \\
		\mathcal{W}^t & \coloneqq \bigl\{s \in L^2(0,T;\mathcal{W}) \,:\, \partial_t s \in L^2(0,T;L^2(\Omega)) \bigr\}.  
	\end{align*}

Testing each equation in problem \eqref{eq:reg_prob}  against suitable functions and integrating by parts whenever adequate, gives the following 
weak formulation: Find $(\rho, \mb{u}, p) \in \mathcal{W}^t \times \mb{V}^t \times \mathcal{Q}^t$ such that for all $(\zeta, \mb{v},q) \in \mathcal{W} \times \mb{V} \times \mathcal{Q}$ and for a.e. $t \in [0, T]$, it holds that
\begin{align}  \label{eq:var_prob} 
	\begin{split} 
		\int_{\Omega}\partial_t \rho \zeta \dx + c_1(\mb{u}, \rho, \zeta) &= 0, \\
\int_{\Omega} \sigma\partial_t  (\sigma\mb{u}) \cdot \mb{v} \dx + a_2(\mb{u},\mb{v}) + c_2(\rho\mb{u};\mb{u},\mb{v}) + b(\mb{v},p) &= \int_{\Omega} \mb{f}\cdot\mb{v}\dx,\\
		b(\mb{u},q) &= 0.
	\end{split}
\end{align}
where as in \cite{Guermond2000}, we use an equivalent equation with $\sigma = \sqrt{\rho}$.  The variational forms $a_2: \mb{V}\times\mb{V} \to \Real$, $c_2: \mathcal{W}\times \mb{V}\times\mb{V} \to \Real$, $c_1: \mb{V}\times\mathcal{W}\times\mathcal{W} \to \Real$ and $b: \mb{V}\times\mathcal{Q}\to \Real$ are defined as follows, for all $\mb{u}, \mb{v}, \mb{w} \in \mb{V}$, $q\in\mathcal{Q}$, $\rho, \zeta \in \mathcal{W}$:

\begin{align*}
	a_2(\mb{u}, \mb{v})  
	& \coloneqq \int_{\Omega} \bigl(\mu(\abs{\mathbf{D} \mb{u}}_{\gamma}) \mathbf{D}\mb{u}:\mathbf{D}\mb{v} \bigr)\dx  ,\\
	c_2(\rho\mb{w}; \mb{u}, \mb{v}) & \coloneqq \int_\Omega (\rho\mb{w} \cdot \nabla)\mb{u}\cdot\mb{v} \dx + \frac{1}{2}\int_{\Omega}\diver(\rho \mb{u}) \mb{u}\cdot\mb{v} \dx,\\
	c_1(\mb{u}; \rho, \zeta) & \coloneqq \int_\Omega (\mb{u} \cdot \nabla \rho)\zeta \dx,\\
	b(\mb{u},q) &= -\int_{\Omega} q \,\diver \mb{u} \dx.
\end{align*}

\cblue{Note that
	\begin{align*}
		\sigma \partial_t (\sigma\mb{u}) = \rho \partial_t(\mb{u}) + \frac{1}{2} \mb{u}\partial_t \rho =\rho \partial_t(\mb{u}) - \frac{1}{2} \nabla \cdot (\rho \mb{u}) \mb{u}
	\end{align*}
	Hence $\sigma \partial_t (\sigma\mb{u}) + ( \rho\mb{u}\cdot\nabla)\mb{u} +\frac{1}{2} \nabla \cdot (\rho \mb{u}) \mb{u}- \Div\bs{\tau}+\nabla p = \mb{f}$, is mathematically equivalent to the original system \eqref{eq:main_prob}. This alternative form of the momentum equation will preserve exactly the kinetic energy balance at the discrete level (see for instance \cite{Guermond2000}).}

\subsection{Stability of the continuous problem}

For the sake of simplicity, we will use homogeneous Dirichlet boundary conditions for velocity in our analysis. Note that more general boundary conditions can still be handled using similar techniques (see, e.g., \cite{Guermond2000}). It is also worth noticing that, for a homogeneous Dirichlet condition on the normal component of the velocity on the entire $\partial \Omega$, no boundary condition needs to be specified for the density.

Before presenting our stability results, we will make some preparatory observations.
\begin{lemma}\label{lem:E<gamma}
	Let $\boldsymbol{\theta},\,\boldsymbol{\vartheta}\in \mathbb{L}^p(\Omega)$. Then, the following estimate holds
	\begin{equation}\label{E<gamma}
		|\boldsymbol{\theta}(x)|_\gamma - |\boldsymbol{\vartheta}(x)|_\gamma\leq \gamma |\boldsymbol{\theta}(x) - \boldsymbol{\vartheta}(x)|,\,\mbox{ a.e. in $\Omega$}.
	\end{equation}
\end{lemma}
\begin{proof}
Let us start defining the following sets, related with the approximations for the yielded and unyielded regions
\begin{equation} \label{eq:zones_approx}
\mathcal{A}_\gamma(\bs{\theta}):=\{x\in \Omega\,:\, \gamma|\bs{\theta}(x)|\geq \tau_s\}\quad\mbox{and}\quad \mathcal{I}_\gamma(\bs{\theta}):=\{x\in \Omega\,:\, \gamma|\bs{\theta}(x)|< \tau_s\}.
\end{equation}
Next, we analyze the behaviour of \eqref{E<gamma} in the following sets $\mathcal{A}_\gamma(\bs{\theta})\cap \mathcal{A}_\gamma(\bs{\vartheta})$, $\mathcal{A}_\gamma(\bs{\theta})\cap \mathcal{I}_\gamma(\bs{\vartheta})$, $\mathcal{I}_\gamma(\bs{\theta})\cap \mathcal{A}_\gamma(\bs{\vartheta})$ and $\mathcal{I}_\gamma(\bs{\theta})\cap \mathcal{I}_\gamma(\bs{\vartheta})$. \\
	
	\noindent \underline{On $\mathcal{A}_\gamma(\bs{\theta})\cap \mathcal{A}_\gamma(\bs{\vartheta})$}: Here, we have that
	\[
	|\boldsymbol{\theta}(x)|_\gamma - |\boldsymbol{\vartheta}(x)|_\gamma= \gamma(|\boldsymbol{\theta}(x)| - |\boldsymbol{\vartheta}(x)|)\leq \gamma  |\boldsymbol{\theta}(x) -\boldsymbol{\vartheta}(x)|.
	\]
	
	\noindent \underline{On $\mathcal{A}_\gamma(\bs{\theta})\cap \mathcal{I}_\gamma(\bs{\vartheta})$,}: Here, it holds that
	\[
	|\boldsymbol{\theta}(x)|_\gamma - |\boldsymbol{\vartheta}(x)|_\gamma= \gamma|\boldsymbol{\theta}(x)| - \tau_s < \gamma|\boldsymbol{\theta}(x)|- \gamma|\boldsymbol{\vartheta}(x)|\leq  \gamma |\boldsymbol{\theta}(x) -\boldsymbol{\vartheta}(x)|.
	\]
	
	\noindent \underline{On $\mathcal{I}_\gamma(\bs{\theta})\cap \mathcal{A}_\gamma(\bs{\vartheta})$}: Here, we know that
	\[
	|\boldsymbol{\theta}(x)|_\gamma - |\boldsymbol{\vartheta}(x)|_\gamma= \tau_s - \gamma|\boldsymbol{\vartheta}(x)| \leq \tau_s-\tau_s=0\leq  \gamma |\boldsymbol{\theta}(x) -\boldsymbol{\vartheta}(x)|.
	\]
	
	\noindent \underline{On $\mathcal{I}_\gamma(\bs{\theta})\cap \mathcal{I}_\gamma(\bs{\vartheta})$}: Here, we obtain the following
	\[
	|\boldsymbol{\theta}(x)|_\gamma - |\boldsymbol{\vartheta}(x)|_\gamma= \tau_s-\tau_s=0\leq  \gamma |\boldsymbol{\theta}(x) - \boldsymbol{\vartheta}|.
	\]
Thus, since the considered sets provide a disjoint partitioning of $\Omega$, the four estimates above imply \eqref{E<gamma}.
\end{proof}

\begin{lemma}\label{lem:prop_mu}
	Let $\boldsymbol{\theta},\,\boldsymbol{\vartheta}\in \mathbb{L}^p(\Omega)$. Then function $\mu$ satisfies the following properties:
	\begin{align}
		|\mu(|\bs{\theta}|_{\gamma})\bs{\theta}-\mu(|\bs{\vartheta}|_{\gamma})\bs{\vartheta}| &\leq C_1|\bs{\theta}-\bs{\vartheta}|, \label{eq:mu_bound}\\
		C_2|\bs{\theta}-\bs{\vartheta}|^2 &\leq (\mu(|\bs{\theta}|_{\gamma})\bs{\theta}-\mu(|\bs{\vartheta}|_{\gamma})\bs{\vartheta}):(\bs{\theta}-\bs{\vartheta}) \nonumber
	\end{align}
	a.e. in $\Omega$
\end{lemma}
\begin{proof}
For the first result, from \eqref{stress_reg} we have:
\begin{align}
	|\mu(|\bs{\theta}|_{\gamma})\bs{\theta}-\mu(|\bs{\vartheta}|_{\gamma})\bs{\vartheta}| &= \abs*{2\eta (\bs{\theta}-\bs{\vartheta}) + \tau_s \gamma\left(\frac{1}{|\bs{\theta}|_{\gamma}}\bs{\theta} - \frac{1}{|\bs{\vartheta}|_{\gamma}}\bs{\vartheta}\right)}  \nonumber \\
	&\leq 2\eta|\bs{\theta}-\bs{\vartheta}| + \tau_s \gamma\abs*{\frac{\bs{\theta}}{|\bs{\theta}|_{\gamma}} - \frac{\bs{\vartheta}}{|\bs{\vartheta}|_{\gamma}}} \label{eq:lem_1}
\end{align}

Following the same lines of the proof of Lemma \ref{E<gamma}, we analyze the behaviour of the second term on the right hand side of \eqref{eq:lem_1} in the following sets $\mathcal{A}_\gamma(\bs{\theta})\cap \mathcal{A}_\gamma(\bs{\vartheta})$, $\mathcal{A}_\gamma(\bs{\theta})\cap \mathcal{I}_\gamma(\bs{\vartheta})$, $\mathcal{I}_\gamma(\bs{\theta})\cap \mathcal{A}_\gamma(\bs{\vartheta})$ and $\mathcal{I}_\gamma(\bs{\theta})\cap \mathcal{I}_\gamma(\bs{\vartheta})$. \\
	
	\noindent \underline{On $\mathcal{A}_\gamma(\bs{\theta})\cap \mathcal{A}_\gamma(\bs{\vartheta})$}: Here, we deduce that
	\begin{align*}
	\tau_s \gamma\abs*{\frac{\bs{\theta}}{|\bs{\theta}|_{\gamma}} - \frac{\bs{\vartheta}}{|\bs{\vartheta}|_{\gamma}}} &= \tau_s\gamma\abs*{\left(\frac{1}{|\bs{\theta}|_{\gamma}}-\frac{1}{|\bs{\vartheta}|_{\gamma}}\right)\bs{\vartheta}+\frac{1}{|\bs{\theta}|_{\gamma}}(\bs{\theta}-\bs{\vartheta})}\\
	&\leq \tau_s\gamma \left(\frac{\gamma|\bs{\vartheta}-\bs{\theta}|}{\gamma|\bs{\theta}|}\frac{|\bs{\vartheta}|}{\gamma|\bs{\vartheta}|}+\frac{1}{\gamma |\bs{\theta}|}|\bs{\theta}-\bs{\vartheta}|\right) \\
	&\leq \tau_s \gamma \frac{2}{\gamma|\bs{\theta}|}|\bs{\theta}-\bs{\vartheta}| \leq \tau_s \gamma \frac{2}{\tau_s}|\bs{\theta}-\bs{\vartheta}|
	\end{align*}
	\noindent \underline{On $\mathcal{A}_\gamma(\bs{\theta})\cap \mathcal{I}_\gamma(\bs{\vartheta})$, (also $\mathcal{I}_\gamma(\bs{\theta})\cap \mathcal{A}_\gamma(\bs{\vartheta})$)}: Here, it follows that
	\begin{align*}
	\tau_s \gamma\abs*{\frac{\bs{\theta}}{|\bs{\theta}|_{\gamma}} - \frac{\bs{\vartheta}}{|\bs{\vartheta}|_{\gamma}}} &= \tau_s\gamma\abs*{\left(\frac{1}{|\bs{\theta}|_{\gamma}}-\frac{1}{|\bs{\vartheta}|_{\gamma}}\right)\bs{\vartheta}+\frac{1}{|\bs{\theta}|_{\gamma}}(\bs{\theta}-\bs{\vartheta})}\\
	&\leq \tau_s\gamma \left(\frac{\tau_s-\gamma|\bs{\theta}|}{\gamma|\bs{\theta}|}\frac{|\bs{\vartheta}|}{\tau_s}+\frac{1}{\gamma |\bs{\theta}|}|\bs{\theta}-\bs{\vartheta}|\right) \\
	&\leq \tau_s\gamma \left(\frac{|\bs{\vartheta}-\bs{\theta}|}{\gamma|\bs{\theta}|}\frac{\tau_s}{\tau_s}+\frac{1}{\gamma |\bs{\theta}|}|\bs{\theta}-\bs{\vartheta}|\right) \\
	&\leq \tau_s \gamma \frac{2}{\gamma|\bs{\theta}|}|\bs{\theta}-\bs{\vartheta}| \leq \tau_s \gamma \frac{2}{\tau_s}|\bs{\theta}-\bs{\vartheta}|
	\end{align*}
	
	\noindent \underline{On $\mathcal{I}_\gamma(\bs{\theta})\cap \mathcal{I}_\gamma(\bs{\vartheta})$}: Here, we obtain
	\begin{align*}
	\tau_s \gamma\abs*{\frac{\bs{\theta}}{|\bs{\theta}|_{\gamma}} - \frac{\bs{\vartheta}}{|\bs{\vartheta}|_{\gamma}}} &= \tau_s\gamma\abs*{\left(\frac{\bs{\theta}}{\tau_s}-\frac{\bs{\vartheta}}{\tau_s}\right)}\\
	&\leq \gamma |\bs{\theta}-\bs{\vartheta}|
	\end{align*}
	The four cases and \eqref{eq:lem_1} imply \eqref{eq:mu_bound}.
	For the second result, we have
	
	\begin{align}
	\left(\mu(|\bs{\theta}|_{\gamma})\bs{\theta}-\mu(|\bs{\vartheta}|_{\gamma})\bs{\vartheta}\right):(\bs{\theta}-\bs{\vartheta}) &= \left({2\eta (\bs{\theta}-\bs{\vartheta}) + \tau_s \gamma\left(\frac{1}{|\bs{\theta}|_{\gamma}}\bs{\theta} - \frac{1}{|\bs{\vartheta}|_{\gamma}}\bs{\vartheta}\right)}\right):(\bs{\theta}-\bs{\vartheta})  \nonumber \\
	&= 2\eta|\bs{\theta}-\bs{\vartheta}|^2 + \tau_s \gamma\left({\frac{\bs{\theta}}{|\bs{\theta}|_{\gamma}} - \frac{\bs{\vartheta}}{|\bs{\vartheta}|_{\gamma}}}\right):(\bs{\theta}-\bs{\vartheta}) \label{eq:lem_2}
\end{align}
Note that the second term in \eqref{eq:lem_2} can be rewritten as
	
	\begin{equation*}
	\begin{array}{lll}
	\gamma\tau_s \left[\left(\frac{1}{|\bs{\theta}|_\gamma} -\frac{1}{|\bs{\psi}|_\gamma}\right) \bs{\vartheta} + \frac{1}{|\bs{\theta}|_\gamma}\left(\bs{\theta}-\bs{\vartheta}\right)\right] :(\bs{\theta}  - \bs{\vartheta})\vspace{0.2cm}\\\hspace{3.4cm}= \gamma\tau_s \left[ \frac{1}{|\bs{\theta}|_\gamma}(\bs{\theta} -\bs{\vartheta}) + \left(\frac{|\bs{\vartheta}|_\gamma - |\bs{\theta}|_\gamma}{|\bs{\vartheta}|_\gamma |\bs{\theta}|_\gamma}\right) \bs{\vartheta}\right] :(\bs{\theta} - \bs{\vartheta}) \vspace{0.2cm}\\\hspace{4.5cm} = \gamma\tau_s \frac{1}{|\bs{\theta}|_\gamma} \left[ |\bs{\theta} -\bs{\vartheta}|^2 - \left(\frac{|\bs{\theta}|_\gamma - |\bs{\vartheta}|_\gamma}{|\bs{\psi}|_\gamma}\right) \bs{\vartheta} :(\bs{\theta} - \bs{\vartheta}) \right].
	\end{array}
	\end{equation*}
Then, thanks to Lemma \ref{lem:E<gamma}, the Cauchy-Schwarz inequality,  and since $\left|\frac{\bs{\vartheta}(x)}{|\bs{\vartheta}(x)|_\gamma}\right|\leq \frac{1}{\gamma}$ a.e. in $\Omega$, we conclude that
	\begin{equation*}
	\begin{array}{lll}
\gamma\tau_s \left[\left(\frac{1}{|\bs{\theta}|_\gamma} -\frac{1}{|\bs{\vartheta}|_\gamma}\right) \bs{\vartheta} + \frac{1}{|\bs{\theta}|_\gamma}\left(\bs{\theta}-\bs{\vartheta}\right)\right] :(\bs{\theta}  - \bs{\vartheta})\vspace{0.2cm}\\\hspace{3.4cm}\geq \gamma\tau_s \frac{1}{|\bs{\theta}|_\gamma} \left[ |\bs{\theta} -\bs{\vartheta}|^2 - \gamma |\bs{\theta} - \bs{\vartheta} |^2 \left|\frac{\bs{\vartheta}(x)}{|\bs{\vartheta}(x)|_\gamma}\right|\right]\geq 0,
	\end{array}
	\end{equation*}	
\end{proof}

Now, recall that the Sobolev embedding Theorem (for instance as in \cite{Adams2003}) establishes the continuous injection $i_r: H^r(\Omega) \to L^{2p}(\Omega)$, where
\begin{align*}
	2p = \begin{cases}
		\frac{1}{1-r} & \text{ if } d=2, \\
		\frac{6}{3-2r} & \text{ if } d=3, \\
	\end{cases}
\end{align*}
and there holds,
\begin{align}
	\norm{v}_{L^{2p}(\Omega)} \leq C_r \norm{v}_{r,\Omega} & \text{ for all }v \in H^r(\Omega). \label{eq:sob_em}
\end{align}

Now we apply the Cauchy-Schwarz and H\"older inequalities, and \eqref{eq:sob_em} to prove that the variational forms defined above are continuous for all 
$\mb{u}, \mb{v}, \mb{w} \in \mb{V}$, $q\in\mathcal{Q}$, and $\rho, \zeta \in \mathcal{W}$:
\begin{subequations}
	\begin{align*}
		\abs[\big]{a_2(\mb{u},\mb{v})} &\leq C_a \norm{\mb{u}}_{1,\Omega} \norm{\mb{v}}_{1,\Omega}, \\
		\abs[\big]{b(\mb{v},q)} &\leq C_b \norm{\mb{v}}_{1,\Omega} \norm{q}_{0,\Omega}, \\
		\abs[\big]{c_2(\rho\mb{w};\mb{u},\mb{v})} &\leq C_c \norm{\rho}_{1,\Omega}\norm{\mb{w}}_{1,\Omega} \norm{\mb{u}}_{1,\Omega} \norm{\mb{v}}_{1,\Omega}, \\ 
		\abs[\big]{c_1(\mb{u};\rho,\zeta)} &\leq \hat{C}_c \norm{\mb{u}}_{1,\Omega} \norm{\rho}_{1,\Omega} \norm{\zeta}_{1,\Omega}.
	\end{align*}
\end{subequations}

We also recall (from \cite[Chapter I, Lemma 3.1]{Girault2005}, for instance) the following Poincar\'e-Friedrichs inequality:
\begin{align} \label{eq:poincare}
	\norm{\varphi}_{0,\Omega} \leq C_p \snorm{\varphi}_{1,\Omega},\qquad \text{for all $\varphi \in H_0^1(\Omega)$}.
\end{align}

Next, we consider the bilinear form $a_2(\cdot,\cdot)$. Note that, for $\mb{v} \in \mb{H}_0^1(\Omega)$, we have that
\[
\begin{array}{lll}
	a_2(\mb{v},\mb{v}) &=&2\eta \int_\Omega \abs{\mathbf{D} \mb{v}}^2\,dx + \gamma\tau_s \int_ \Omega \frac{\abs{\mathbf{D}\mb{v}}^2}{\abs{\mathbf{D}\mb{v}}_\gamma}\,dx\vspace{0.2cm}\\ &\geq&  C \norm{\mathbf{D} \mb{v}}^2_{0,\Omega}\,dx.
\end{array}
\]
Hence, Korn's inequality and inequality \eqref{eq:poincare} readily gives the coercivity of $a_2$, i.e., there exists a positive constant~$\alpha_a$ such that  
\begin{subequations}
	\begin{align}
		a_2(\mb{v},\mb{v}) & \geq \alpha_a \norm{\mb{v}}^2_{1,\Omega}, \qquad \text{for all $\mb{v} \in \mb{H}^1_0({\Omega})$.}  \label{eq:coerciv_a}
\end{align}\end{subequations}
Using the definition and characterisation of the kernel of the bilinear form $b(\cdot,\cdot)$, we can write 
\begin{align*} 
	\mb{K}\coloneqq\bigl\{ \mb{v} \in \mb{H}^1_0(\Omega)\, : \, b(\mb{v},q)=0, \quad \forall q \in L_0^2(\Omega) \bigr\} = \bigl\{ \mb{v} \in \mb{H}_0^1(\Omega)
	\, : \, \diver \mb{v} = 0 \; \text{in $\Omega$} \bigr\},   
\end{align*}
and applying integration by parts, we can readily observe that (see \cite[Section 2.2]{Li2013}, \cite[Lemma 1]{Guermond2000}, also \cite{Pyo2007}) 
\begin{align} 
	\text{$c_2(\rho\mb{w};\mb{v},\mb{v}) = 0$  \quad and  \quad $c_1(\mb{w};\rho,\rho) = 0$, \quad for all  
		$\mb{w} \in \mb{K}, \mb{v} \in \mb{H}^1(\Omega),  \rho \in H^1(\Omega)$.} 
	\label{eq:c_zero}
\end{align}

Finally, it is well known that the bilinear form $b(\cdot,\cdot)$ satisfies the inf-sup condition (see, e.g., \cite{temam77}):
\begin{align*}
	\sup_{\mb{v} \in \mb{H}_0^1(\Omega) \backslash \{\mb{0}\}} \frac{b(\mb{v},q)}{\norm{\mb{v}}_{1,\Omega}} &\geq \beta \norm{q}_{0,\Omega}, 
	\quad \text{for all $q \in L_0^2(\Omega)$}.   
\end{align*}

\begin{lemma}[Stability]If $\mb{f} \in L^{\infty}(0,T;\mb{L}^{2}(\Omega))$, $\mb{u}_0 \in \mb{L}^2(\Omega)$ and $\rho_0 \in L^2(\Omega)$, then, for any solution $(\mb{u}, \rho)$ of \eqref{eq:var_prob} and for $t \in (0,T]$, there exists a constant $C>0$ such that
	\begin{align*}
		\norm{\sigma\mb{u}}_{L^2(0,t;\mb{H}^1(\Omega))} + \norm{\rho}_{L^2(0,t;H^1(\Omega))} &\leq C \bigl(\norm{\sqrt{\rho_0}\mb{u}_0}_{0,\Omega} + \norm{\rho_0}_{0,\Omega} + \norm{\mb{f}}_{L^{\infty}(0,T;\mb{L}^2(\Omega))}\bigr),
	\end{align*}
\end{lemma}
\begin{proof}
	First taking $\zeta = \rho$, in the first equation of \eqref{eq:var_prob} and using \eqref{eq:c_zero}, we obtain the following identity:
	\begin{align*}
		\frac{1}{2} \frac{\mathrm{d}}{\mathrm{d} t}\norm{\rho}_{0,\Omega}^2 = 0
	\end{align*}
	Integrating this equation between $0$ and $t$ yields, in particular, that
	\begin{align}\label{eq:res_stab_1}
		\norm{\rho(\cdot,t)}_{0,\Omega} \leq \norm{\rho_0}_{0,\Omega}. 
	\end{align}

	For the momentum equation, we first deduce
	\begin{align*}
		\int_{\Omega} \rho \mb{u} \cdot \partial_t \mb{u} \dx &= \frac{1}{2} \int_{\Omega} \rho \frac{\partial |\mb{u}|^2}{\partial t} \dx = \frac{1}{2} \int_{\Omega} \biggl( \frac{\partial(\rho|\mb{u}|^2)}{\partial t}+(\mb{u} \cdot\nabla \rho+\frac{\rho}{2}\nabla \cdot \mb{u})|\mb{u}|^2\biggr) \dx \\
		\int_{\Omega} \rho \mb{u} \cdot \nabla \mb{u}\cdot \mb{u} \dx &= \frac{1}{2} \int_{\Omega}\rho \mb{u} \cdot \nabla |\mb{u}|^2 \dx = \frac{1}{2}\int_{\Omega} \biggl(\nabla \cdot (\rho \mb{u} |\mb{u}|^2) - (\mb{u}\cdot \nabla \rho + \rho \nabla \cdot \mb{u})|\mb{u}|^2 \biggr)\dx
	\end{align*}
	which in turn, implies,
	\begin{align} \label{eq:dt-u}
		\mb{u} \cdot \biggl[ \rho \partial_t \mb{u} +  \rho \mb{u} \cdot \nabla \mb{u} \biggr] = \frac{1}{2} \biggl[ \frac{(\rho|\mb{u}|^2)}{\partial t} + \nabla \cdot (\rho \mb{u} |\mb{u}|^2) \biggr]
	\end{align}
	Now, we can take $\mb{u}$ on $\mb{K}$ and due to the inf-sup condition we can solve an equivalent reduced problem, where $b(\cdot,\cdot)$ is removed from the variational form \eqref{eq:var_prob}. Setting $\mb{v}=\mb{u}$, using the incompressibility condition, \eqref{eq:c_zero}, \eqref{eq:coerciv_a} and \eqref{eq:dt-u}, we have

	\begin{align*}
		\frac{1}{2} \frac{\mathrm{d}}{\mathrm{d} t}\norm{\sigma\mb{u}}_{0,\Omega}^2 + \alpha_a \norm{\mb{u}}_{1,\Omega}^2 \leq \norm{\mb{f}}_{0,\Omega}\norm{\mb{u}}_{0,\Omega}.
	\end{align*}
	Now we use Young's  inequality with $\varepsilon = \alpha_a/4$ to get
	\begin{align*}
		\frac{1}{2} \frac{\mathrm{d}}{\mathrm{d} t}\norm{\sigma\mb{u}}_{0,\Omega}^2 + \frac{\alpha_a}{2} \norm{\mb{u}}_{1,\Omega}^2 \leq C\norm{\mb{f}}_{0,\Omega}^2.
	\end{align*}
	Analogously, after integrating from~$0$ to~$t$ we find that
	\begin{align} \label{eq:res_stab_2}
		\norm{\sigma\mb{u}(\cdot,t)}_{0,\Omega} + \alpha_a \int_0^t \norm{\mb{u}(\cdot,z)}_{1,\Omega}^2 \dz \leq \norm{\sqrt{\rho_0} \mb{u}_0}_{0,\Omega} + C\int_{0}^t\norm{\mb{f}}_{0,\Omega}\dz.
	\end{align}
Finally, we derive the sought result from \eqref{eq:res_stab_1} and \eqref{eq:res_stab_2}.
\end{proof}

\section{The discrete formulation} \label{sec:Discr}

In this section we introduce the Galerkin scheme associated with problem \eqref{eq:var_prob}.

\subsection{The time semi-discrete problem}
In order to describe the time discretization of equation \eqref{eq:main_prob}, we introduce a partition of the interval $[0, T ]$ into subintervals $[t_{n-1}, t_n]$, $1 \leq n \leq N$,  such that $0 = t_0 < t_1 < \dots < t_N = T$. We  use the implicit BDF2 scheme, where all first-order time derivatives are approximated using the centered operator  
\begin{align} \label{centrop} 
	\partial_t \mb{u} (t_{n+1})\approx \frac{1}{\Dt}\biggl(\frac{3}{2}\mb{u}^{n+1}-2\mb{u}^n
	+\frac{1}{2}\mb{u}^{n-1}\biggr), \end{align} 
(similarly for $\partial_t \rho$), and for the first time step a first-order backward Euler method is used 
from $t^0$ to $t^1$, starting from the interpolates $\mb{u}^0$ and $\rho^0$ of the initial data.  

In what follows, we define the difference operator 
\[ 
\mathcal{D} y^{n+1}  \coloneqq  3 y^{n+1} - 4 y^n   + y^{n-1}, 
\]
for any quantity indexed by the time step $n$.  For instance,  \eqref{centrop} can be written as 
$\partial_t \mb{u}(t_{n+1})\approx  \frac{1}{2 \Delta t}  \mathcal{D} \mb{u}^{n+1}$. 

In the following sections, we discuss a dG-$H(\mathrm{div})$-FEM discretization for the space variables, and present the fully discretized system to be solved by a semismooth Newton iteration. 

\subsection{A Divergence-conforming-dG FEM coupled scheme}

Let $\{\bs{\mathcal{T}}_h\}_{h>0}$ be a regular family of triangulations of $\Omega$ by simplices $K$ (triangles in $\mathbb{R}^2$ and  tetrahedra in $\mathbb{R}^3$ respectively), and set $h := \max \{h_K : K \in \bs{\mathcal{T}}_h\}$, where $h_k$ is the diameter of the element $K$. 
We   label by  $K^-$ and $K^+$ the two elements adjacent to a facet $e$ (an edge in 2D or a face in 3D). Let $\Eh$  denote the set of all facets and $\Eh = \Eh^i \cup \Eh^{\partial}$ where $\Eh^i$ and $\Eh^{\partial}$ are the subset of interior facets and boundary facets, respectively.  If $\mb{v}$ and $w$ are a smooth vector and a scalar field defined on $\{\bs{\mathcal{T}}_h\}$,  then  ($\mb{v}^\pm, w^\pm$)  denote the traces of ($\mb{v},w$) on $e$ 
that are  the extensions from the interior of $K^+$ and $K^-$, respectively.  Let $\mb{n}^+_e$, $\mb{n}^-_e$~be the outward unit normal vectors on the boundaries of two neighboring elements sharing the facet $e$, $K^+$ and~$K^-$, respectively. We also use the notation $(\mb{w}_e\cdot\mb{n}_e)|_{e}=(\mb{w}^{+}\cdot\mb{n}^{+}_e)|_{e}$. 
The average $\dmean{\cdot}$ and jump $\djump{\cdot }$ operators on $e\in \Eh^{i}$ are defined as 
\begin{align*}
	\dmean{\mb{v}} &\coloneqq (\mb{v}^-+\mb{v}^+)/2, \quad \dmean{w} \coloneqq  (w^-+w^+)/2, \quad 
	\djump{\mb{v} } \coloneqq  (\mb{v}^- - \mb{v}^+), \quad \djump{w } \coloneqq  (w^- - w^+),
\end{align*}
whereas, for jumps and averages on $e\in \Eh^{\partial}$, for notational convenience, we adopt the conventions  $\dmean{\mb{v}} = \djump{\mb{v}} = \mb{v}$,  and $\dmean{w} = \djump{w} = w$.  Moreover,  $\mathbf{D}_h$ will  denote  the broken analogous of operator $\mathbf{D}$.

For $k\geq 1$, consider the following finite element subspaces:
\begin{align*}
	\mb{V}_h &\coloneqq \bigl\{ \mb{v}_h \in \mb{H}(\Div; \Omega): \mb{v}_h|_K \in [\mathbb{P}^{k}(K)]^{\mathrm{d}}\quad\forall K \in \mathcal{T}_h \bigr\}, \\
	\mathcal{Q}_h &\coloneqq \bigl\{ q_h \in L^2_0(\Omega) : q_h|_K \in \mathbb{P}^{k-1}(K)\quad\forall K \in \mathcal{T}_h\bigr\}, \\
	\mathcal{W}_h &\coloneqq \bigl\{ s_h \in L^2({\Omega}) : l_h|_K \in \mathbb{P}^{k-1}(K) \quad\forall K \in \mathcal{T}_h \bigr\}.
\end{align*}

Associated with these finite-dimensional spaces, we state the following semi-discrete Galerkin formulation: Find $(\rho_h, \mb{u}_h, p_h) \in \mathcal{W}_h\times \mb{V}_h\times \mathcal{Q}_h$, such that, for all $(\zeta_h, \mb{v}_h,q_h) \in \mathcal{W}_{h} \times \mb{V}_h \times \mathcal{Q}_h$, it holds that
\begin{align*}
	\begin{split} 
		(\partial_t \rho_h, \zeta_h)_{\Omega} + c_1^h(\mb{u}_h, \rho_h, \zeta_h) &= 0, \\
		(\sigma_h \partial_t(\sigma_h \mb{u}_h), \mb{v}_h)_{\Omega} + a^h_2(\mb{u}_h,\mb{v}_h) + c^h_2(\rho_h\mb{u}_h;\mb{u}_h,\mb{v}_h) + b(p_h,\mb{v}_h) &= ( \mb{f},\mb{v}_h)_{\Omega},\\
		b(q_h, \mb{u}_h) &= 0.
	\end{split}
\end{align*}
Here $\sigma_h = \sqrt{\rho_h}$. Moreover, the discrete versions of the forms $a_2^h(\cdot,\cdot)$, $c_2^h(\cdot,\cdot;\cdot,\cdot)$ and $c_1^h(\cdot;\cdot,\cdot)$ are defined by using a symmetric interior penalty approach in the first case and upwind approach for the two convective terms:
\begin{equation}
\begin{split}
a^h_2(\mb{u}_h, \mb{v}_h)  
	& \coloneqq \int_{\Omega} \bigl(\mu(\abs{\mathbf{D}_h \mb{u}_h}_{\gamma}) \mathbf{D}_h\mb{u}_h:\mathbf{D}_h\mb{v}_h \bigr)\dx   \label{eq:a2hdef}\\
	&  \quad + \sum_{e\in\mathcal{E}_h} \int_e \biggl( -\dmean{\mu(\abs{\mathbf{D}_h \mb{u}_h}_{\gamma}) \mathbf{D}_h(\mb{u}_h)\mb{n}_e} \cdot \djump{\mb{v}_h }  \\
	& - \dmean{\mu(\abs{\djump{\mathbf{D}_h \mb{u}_h}}_{\gamma}) \mathbf{D}_h(\mb{v}_h)\mb{n}_e}\cdot\djump{\mb{u}_h } + \frac{a_0}{h_e} \djump{\mb{u}_h } : \djump{\mb{v}_h } \biggr) \ds,
\end{split}
\end{equation}
\begin{equation*}
\begin{split}
c^h_1 (\mb{u}_h; \rho_h, \zeta_h) & \coloneqq \int_\Omega (\mb{u}_h \cdot \nabla \rho_h)\zeta_h \dx - \sum_{e\in\mathcal{E}_h} \int_e (\mb{u}_h\cdot \mb{n}_e) \djump{\rho_h}\dmean{\zeta_h}\ds \\
	& \quad + \sum_{e\in\mathcal{E}_h} \int_e \frac{1}{2}\abs{\mb{u}_h\cdot \mb{n}_e} \djump{\rho_h}\cdot\djump{\zeta_h}\ds,\\
\end{split}
\end{equation*}
\begin{equation*}
\begin{split}
c^h_2 (\rho_h \mb{w}_h; \mb{u}_h, \mb{v}_h) & \coloneqq c_2(\rho_h \mb{w}_h; \mb{u}_h, \mb{v}_h)  \\& \quad - \sum_{e\in\mathcal{E}_h} \int_e (\rho_h\mb{w}_h\cdot \mb{n}_e) \djump{\mb{u}_h}\cdot\dmean{\mb{v}_h}\ds
	 + \sum_{e\in\mathcal{E}_h} \int_e \frac{1}{2}\abs{\rho_h\mb{w}_h\cdot \mb{n}_e} \djump{\mb{u}_h}\cdot\djump{\mb{v}_h}\ds,
\end{split}
\end{equation*}
where $a_0>0$ is a jump penalization parameter.

\subsection{Complete discrete scheme}
We now define the approximate sequences $\{\rho^n_h\}_{n=0,\dots N}$, $\{\mb{u}_h^n\}_{n=0,\dots N}$ and $\{p^n_h\}_{n=0,\dots N}$ as follows: For $1 \leq n \leq N-1$, solve:

\begin{align}  \label{eq:full_disc_scheme} 
	\begin{split} 
		\frac{1}{2 \Dt} \bigl(\mathcal{D} \rho_h^{n+1}, \zeta_h \bigr)_{\Omega} + c_1^h(\mb{u}_h^{n+1}, \rho_h^{n+1}, \zeta_h) &= 0, \\
		\frac{1}{2\Dt} \bigl( \sigma_h^{n+1} \mathcal{D} (\sigma_h^{n+1} \mb{u}_h^{n+1}\bigr), \mb{v}_h)_{\Omega} + a^h_2(\mb{u}_h^{n+1},\mb{v}_h)   \\\hspace{2cm}+ c^h_2(\rho_h^{n+1}\mb{u}_h^{n+1};\mb{u}_h^{n+1},\mb{v}_h)
		+ b(p_h^{n+1},\mb{v}_h) &= (\mb{f}^{n+1},\mb{v}_h)_{\Omega},\\
		b(q_h,\mb{u}_h^{n+1}) &= 0.
	\end{split}
\end{align}
for all $\zeta_h \in \mathcal{W}_h$, $\mb{v}_h \in \mb{V}_h$ and $q_h \in \mathcal{Q}_h$. Note that non-homogeneous Dirichlet boundary conditions for the velocity field can be imposed as part of this formulation by using Nitsche's method.

\subsection{Stability analysis of the discrete scheme}\label{sec:stab}

For the subsequent analysis, we introduce,  for $r \geq 0$, the broken $\mb{H}^r$ space as follows.
\[
\mb{H}^r(\Th) \coloneqq \bigl\{ \mb{v} \in \mb{L}^2(\Omega) : \mb{v}|_K \in \mb{H}^r(K),\ K \in \Th \bigr\}, 
\]
as well as the mesh-dependent broken norms
\begin{gather*}
	\norm{\mb{v}}_{*,\Th}^2  \coloneqq \sum_{K\in \Th} \norm{\mathbf{D}_h(\mb{v})}_{0,K}^2 + \sum_{e\in\Eh} \frac{1}{h_e} \norm{ \djump{\mb{v} } }_{0,e}^2, \\ 
	\norm{\mb{v}}_{\Thnorm}^2   \coloneqq \norm{\mb{v}}_{0,\Omega}^2 + \norm{\mb{v}}_{*,\Th}^2 \quad \text{for all $\mb{v} \in \mb{H}^1(\Th)$,}  
	\end{gather*}
We also define the discrete kernel of the bilinear form $b(\cdot,\cdot)$ as 
\begin{align*}
	\mb{K}_h \coloneqq \{ \mb{v}_h\in \mb{V}_h\,:\,b(\mb{v}_h,q_h) = 0\; \forall q_h \in  \mathcal{Q}_h\} = \{ \mb{v}_h\in\mb{V}_h\,:\,\diver \mb{v}_h = 0 \text{ in  $\Omega$}\}.
\end{align*}
Finally, adapting the argument used in \cite[Proposition 4.5]{karakashian1998}, we have the discrete Sobolev embedding:
for $r=2,4$ there exists a constant $C_{\textnormal{emb}}>0$ such that
\begin{align*}
	\norm{\mb{v}}_{\mb{L}^{r}(\Omega)} \leq C_{\textnormal{emb}} \norm{\mb{v}}_{\Thnorm}, \qquad \text{for all }\mb{v} \in \mb{H}^{1}(\Th).
\end{align*}
With these norms, we can establish continuity of the 
bilinear forms constituting the variational formulation. The proof  follows from   \cite[Section 4]{arnold2001} and \cite[Lemma 2.2]{Houston2005}. 
\begin{lemma} \label{lemma4:ab_bounds}
	The following properties hold:
	\begin{align*}
		\abs[\big]{a^h_2(\mb{u},\mb{v})}  &\leq \tilde{C}_a \norm{\mb{u}}_{\Thnorm}\norm{\mb{v}}_{\Thnorm},    &&\text{for all $\mb{u},\mb{v}\in \mb{V}_h$,}  \\
		\abs[\big]{b(\mb{v},q)}  &\leq \tilde{C}_b \norm{\mb{v}}_{\Thnorm}\norm{q}_{0,\Omega},   &&\text{for all $ \mb{v} \in \mb{H}^1(\Th)$, $q\in L^2(\Omega)$.}  
	\end{align*} 
\end{lemma}

Furthermore the following property holds (using Lemma \ref{lem:prop_mu} and following arguments analogous to those in \cite[Lemma 2.3]{Houston2005}, see also \cite[Theorem 2.4]{Congreve2013})
\begin{equation}
	a^h_2(\mb{v},\mb{v}) \geq \tilde{\alpha}_a \norm{\mb{v}}_{\Thnorm}^2 \quad \text{for all $\mb{v}\in \mb{V}_h$,} 
	\label{eq:coerciv_ah}
\end{equation}
provided that the stabilization parameter $a_0>0$ in \eqref{eq:a2hdef} is sufficiently large and independent of the mesh size. 

Let $\mb{w} \in \mb{H}_0(\Div^0;\Omega)$ and let us introduce the following vector and scalar  jump seminorms
	\begin{align*}
		\snorm{\mb{u}_h}_{\mb{w},\mathrm{upw}} \coloneqq \sum_{e \in \Eh^i}\int_{e} \frac{1}{2}\abs{\mb{w}_e \cdot \mb{n}_e} \abs{\djump{\mb{u}_h}}^2 \ds, \\
		\snorm{\rho_h}_{\mb{w},\mathrm{upw}} \coloneqq \sum_{e \in \Eh^i}\int_{e} \frac{1}{2}\abs{\mb{w}_e \cdot \mb{n}_e} \abs{\djump{\rho_h}}^2 \ds. 
\end{align*}
Then, due to the skew-symmetric form of the operators~$c_1^h$ and~$c_2^h$, and the positivity of the non-linear upwind terms (see i.e \cite{Pyo2007} and \cite[Section 2.3.1]{DiPietro2012}), we can write
\begin{subequations}\begin{align}
				c_{1}^h(\mb{w};\psi_h,\psi_h) &= \snorm{\psi_h}_{\mb{w},\mathrm{upw}}^2 \geq0 \quad \text{for all $\psi_h \in \mathcal{W}_{h}$,}
		\label{eq:crhoeq0}\\
		c^h_2(\rho\mb{w};\mb{u},\mb{u})  &= \snorm{\mb{u}}_{\rho\mb{w},\mathrm{upw}}^2 \geq0 \quad \text{for all $\mb{u}\in \mb{V}_h$.} \label{eq:cueq0}
\end{align}\end{subequations}
Finally, we recall from \cite{konno2011} the following discrete inf-sup condition for $b(\cdot,\cdot)$, where $\tilde{\beta}$ is independent of~$h$: 
\begin{equation*}
	\sup_{\mb{v}_h\in \mb{V}_h\backslash\{\mb{0}\}} \frac{b(\mb{v}_h,q_h)}{\norm{ \mb{v}_h}_{1,\Th}} \geq 
	\tilde{\beta} \norm{q_h}_{0,\Omega}, \quad \text{for all $q_h \in \mathcal{Q}_h$.}  
\end{equation*}


\begin{theorem}
	Let $\smash{(\rho_h^{n+1},\mb{u}_h^{n+1}, p_h^{n+1}) \in \mathcal{W}_h\times\mb{V}_h \times \mathcal{Q}_h}$ be a solution of problem \eqref{eq:full_disc_scheme}, {with initial data $(\rho_h^1,\mb{u}_h^1)$ and $(\rho_h^0,\mb{u}_h^0)$}. Then the following bounds are satisfied, where $C_1$ and $C_2$ are constants that are independent of $h$ and $\Dt$:
	\begin{equation} \label{eq:utcs_bounds}
		\begin{array}{lll}
			\norm{\rho_h^{n+1}}^2_{0,\Omega} + \norm{2\rho^{n+1}_h - \rho_h^{n}}^2_{0,\Omega} + \sum_{j=1}^{n} \norm{\Lambda \rho_h^{j}}^2_{0,\Omega} + \sum_{j=1}^{n} \Dt \snorm{\rho^{j+1}_h}^2_{\mb{u}_h^{j},\upw}\vspace{0.2cm}\\\hspace{6cm}\leq C_1 \bigl( \norm{\rho_h^{1}}^2_{0,\Omega} + \norm{2\rho_h^{1} - \rho_h^{0}}^2_{0,\Omega} \bigr), \vspace{0.3cm}\\
			\norm{\sigma_h^{n+1}\mb{u}_h^{n+1}}^2_{0,\Omega} + \norm{2\sigma_h^{n+1}\mb{u}^{n+1}_h - \sigma_h^{n}\mb{u}_h^n}^2_{0,\Omega} +  \sum_{j=1}^{n} \norm{\Lambda \mb{u}_h^j}^2_{0,\Omega}  + \sum_{j=1}^{n} \Dt \norm{\mb{u}^{j+1}_h}^2_{\Thnorm}\vspace{0.2cm}\\\hspace{1cm} + { \sum_{j=1}^{n} \Dt \snorm{\mb{u}_h^j}^2_{\mb{u}_h^j,\mathrm{upw}}} 
			 \leq C_2 \bigl(\norm{\mb{f}}^2_{L^{\infty}(0,T,\mb{L}^2(\Omega))} + \norm{\sigma_h^{1}\mb{u}_h^1}^2_{0,\Omega} + \norm{2\sigma_h^{1}\mb{u}_h^1 - \sigma_h^{0}\mb{u}_h^0}^2_{0,\Omega}  \bigr).
	\end{array}
	\end{equation}
\end{theorem}
\begin{proof}
	We will require the following algebraic relation: for any real positive numbers $a^{n+1}$, $a^n$, $a^{n-1}$ and 
	defining $\Lambda a^n \coloneqq  a^{n+1} - 2a^n + a^{n-1}$, we have
	\begin{align} \label{eq:alg_id} 
		2(3a^{n+1} - 4a^n + a^{n-1}, a^{n+1}) &= \abs{a^{n+1}}^2 + \abs{2a^{n+1}-a^n}^2 + \abs{\Lambda a^n}^2 
		- \abs{a^{n}}^2 - \abs{2a^n-a^{n-1}}^2. 
	\end{align}
	First we take $\zeta_h = 4\rho_h^{n+1}$ in the first equation of \eqref{eq:full_disc_scheme}, multiply by $\Dt$ and apply \eqref{eq:alg_id} and \eqref{eq:crhoeq0}  to deduce the estimate
	\begin{align*}
		& \norm{\rho_h^{n+1}}^2_{0,\Omega} + \norm{2\rho^{n+1}_h - \rho_h^{n}}^2_{0,\Omega} + \norm{\Lambda \rho_h^{n}}^2_{0,\Omega} + 4\Dt \snorm{\rho^{n+1}_h}^2_{\mb{u}^{n+1}_h,\upw}  \leq \norm{\rho_h^{n}}^2_{0,\Omega} + \norm{2\rho_h^{n} - \rho_h^{n-1}}^2_{0,\Omega}.
	\end{align*}
	Summing over $n$ we can assert that
		\begin{align*} \begin{split} 
			&	\norm{\rho_h^{n+1}}^2_{0,\Omega} + \norm{2\rho^{n+1}_h - \rho_h^{n}}^2_{0,\Omega} + \sum_{j=1}^{n} \norm{\Lambda \rho_h^{j}}^2_{0,\Omega} + \sum_{j=1}^{n} \Dt \snorm{\rho^{j+1}_h}^2_{\mb{u}_h^{j+1},\upw}  \leq \norm{\rho_h^{1}}^2_{0,\Omega} + \norm{2\rho_h^{1} - \rho_h^{0}}^2_{0,\Omega}.
		\end{split} 
	\end{align*}
	
	Similarly in the second and third equation of \eqref{eq:full_disc_scheme}, we take $\mb{v}_h = 4\mb{u}_h^{n+1}$ and $q_h = 4p_h^{n+1}$, respectively, multiply by $\Dt$ and apply \eqref{eq:alg_id}, \eqref{eq:coerciv_ah} and \eqref{eq:cueq0} to deduce the estimate
	\begin{align*}
		&	\norm{\sigma_h^{n+1}\mb{u}_h^{n+1}}^2_{0,\Omega} + \norm{2\sigma_h^{n+1}\mb{u}^{n+1}_h - \sigma_h^{n}\mb{u}_h^n}^2_{0,\Omega} + \norm{\Lambda \sigma_h^{n}\mb{u}_h^n}^2_{0,\Omega} + 4 \Dt \tilde{\alpha}_a \norm{\mb{u}^{n+1}_h}^2_{\Thnorm} + { 4 \Dt \snorm{\mb{u}_h^{n+1}}^2_{\mb{u}_h^{n+1},\mathrm{upw}}}
		\\ & \quad 
		\leq 4 \Dt\norm{\mb{f}^{n+1}}_{0,\Omega} \norm{\mb{u}_h^{n+1}}_{0,\Omega} + \norm{\mb{u}_h^n}^2_{0,\Omega} + \norm{2\mb{u}_h^n - \mb{u}_h^{n-1}}^2_{0,\Omega}.
	\end{align*}
	Using Young's inequality with $\varepsilon = \tilde{\alpha}_a/2$ and summing over $n$ we can assert that
	\begin{align*}   \begin{split} 
			&\norm{\sigma_h^{n+1}\mb{u}_h^{n+1}}^2_{0,\Omega} + \norm{2\sigma_h^{n+1}\mb{u}^{n+1}_h - \sigma_h^{n}\mb{u}_h^n}^2_{0,\Omega} + \sum_{j=1}^{n} \norm{\Lambda \sigma_h^{j}\mb{u}_h^j}^2_{0,\Omega} \\ &\qquad + 2 \tilde{\alpha}_a \sum_{j=1}^{n} \Dt \norm{\mb{u}^{j+1}_h}^2_{\Thnorm} + { \sum_{j=1}^{n}4 \Dt \snorm{\mb{u}_h^{j+1}}^2_{\rho_h^{n+1}\mb{u}_h^{j+1},\mathrm{upw}}} \\ &\qquad \leq  C\norm{\mb{f}}_{L^{\infty}(0,T,\mb{L}^2(\Omega))} + \norm{\sigma_h^1\mb{u}_h^1}^2_{0,\Omega} + \norm{2\sigma_h^1\mb{u}_h^1 - \sigma_h^{0}\mb{u}_h^{0}}^2_{0,\Omega}.  \end{split} 
	\end{align*}
\end{proof}
\begin{theorem}[Existence of discrete solutions] \label{th:disc-exist} 
	Problem 	\eqref{eq:full_disc_scheme} with initial data $(\rho_h^1, \mb{u}_h^1)$ and $(\rho_h^0,\mb{u}_h^0)$ admits  at least one solution 
	\begin{align*} 
		(\rho_h^{n+1},\mb{u}_h^{n+1}, p_h^{n+1}) \in \mathcal{W}_h\times\mb{V}_h \times \mathcal{Q}_h. \end{align*} 
\end{theorem}

{The proof of Theorem~\ref{th:disc-exist} makes use of Brouwer's fixed-point} theorem in the following form 
(given by \cite[Corollary 1.1, Chapter IV]{girault1986}): 

\begin{theorem}[Brouwer's fixed-point theorem] \label{th:brouwer}  Let $H$ be a finite-dimensional Hilbert space with scalar product  $(\cdot,\cdot )_H$ and corresponding norm $\norm{\cdot}_H$. Let $\Phi\colon H \to H$ be a continuous mapping  for which there exists $\vartheta > 0$ such that $(\Phi(u),u)_H \geq 0$ for all $u \in H$  with $\norm{u}_H  = \vartheta$. 		Then  there exists $u \in H$ such that $\Phi(u) = 0$ and $\norm{u}_H \leq \vartheta$. \end{theorem} 

\begin{proof}[Proof of Theorem~\ref{th:disc-exist}]
	To simplify the proof we introduce the  constants 
	\begin{gather*} 
		C_{\rho}  \coloneqq  C_1 \bigl( \norm{\rho_h^{1}}^2_{0,\Omega} + \norm{2\rho_h^{1} - \rho_h^{0}}^2_{0,\Omega} \bigr), \\
		C_{u}  \coloneqq C_2 \bigl(\norm{\mb{f}}^2_{L^{\infty}(0,T,\mb{L}^2(\Omega))} + \norm{\sigma_h^{1}\mb{u}_h^1}^2_{0,\Omega} + \norm{2\sigma_h^{1}\mb{u}_h^1 - \sigma_h^{0}\mb{u}_h^0}^2_{0,\Omega}  \bigr).
	\end{gather*}
	We proceed by induction on $n \geq 2$. We define the mapping 
	\begin{equation}\label{eq:fixed} 
		\Phi : \mathcal{W}_h\times\mb{V}_h \times \mathcal{Q}_h \to \mathcal{W}_h\times\mb{V}_h \times \mathcal{Q}_h,  
	\end{equation} 
	using the relation
	\begin{align*}
		&\bigl( \Phi(\rho_h^{n+1},\mb{u}_h^{n+1}, p_h^{n+1}),(\zeta_h,\mb{v}_h, q_h) \bigr)_{\Omega} \\
		&\hspace{1cm} = \frac{1}{2 \Dt} \bigl(\mathcal{D} \rho_h^{n+1}, \zeta_h \bigr)_{\Omega} + c_1^h(\mb{u}_h^{n+1}, \rho_h^{n+1}, \zeta_h) + \\
		&\hspace{1.5cm} \frac{1}{2\Dt} \bigl( \sigma_h^{n+1} \mathcal{D}_h (\sigma_h^{n+1} \mb{u}_h^{n+1}\bigr), \mb{v}_h)_{\Omega} + a^h_2(\mb{u}_h^{n+1},\mb{v}_h)   \\ 
		&\hspace{1.5cm}+ c^h_2(\rho_h^{n+1}\mb{u}_h^{n+1};\mb{u}_h^{n+1},\mb{v}_h)
		+ b(p_h^{n+1},\mb{v}_h) - (\mb{f}^{n+1},\mb{v}_h)_{\Omega} \\
		&\hspace{1.5cm}- b(q_h,\mb{u}_h^{n+1}).
	\end{align*}
	Note that this map is well-defined and continuous on $\mathcal{W}_h\times\mb{V}_h \times \mathcal{Q}_h$. On the other hand, if we take 
	$$(\zeta_h,\mb{v}_h, q_h)=\left(\rho_h^{n+1},\mb{u}_h^{n+1}, p_h^{n+1}\right),$$ 
	and employ \eqref{eq:crhoeq0}, \eqref{eq:cueq0}, and \eqref{eq:coerciv_ah}, we obtain
	\begin{align*}
		&\bigl( \Phi(\rho_h^{n+1},\mb{u}_h^{n+1}, p_h^{n+1}),(\rho_h^{n+1},\mb{u}_h^{n+1}, p_h^{n+1}) \bigr)_{\Omega}\\
		& \geq \frac{3}{2\Dt} \norm{\rho_h^{n+1}}^2_{0,\Omega}-\frac{1}{2\Dt}(4\rho_h^n-\rho_h^{n-1},\rho_h^{n+1})_{\Omega}+\snorm{\rho_h^{n+1}}_{\mb{u}_h^{n+1},\mathrm{upw}}
		\\&+\frac{3}{2\Dt}\norm{\mb{u}_h^{n+1}}_{0,\Omega}^2-\frac{1}{2\Dt} (4\mb{u}_h^n-\mb{u}_h^{n-1},\mb{u}_h^{n+1})_{\Omega} + \tilde{\alpha}_a\norm{\mb{u}_h^{n+1}}^2_{\Thnorm}\\ &+{\snorm{\mb{u}_h^{n+1}}_{\mb{u}_h^{n+1},\mathrm{upw}}^2}  - (\mb{f}^{n+1},\mb{u}_h^{n+1})_{\Omega}.
	\end{align*}
	Next,   using \eqref{eq:utcs_bounds} and Cauchy-Schwarz inequality, we deduce that
	\begin{align*}
		&\bigl( \Phi(\rho_h^{n+1},\mb{u}_h^{n+1}, p_h^{n+1}),(\rho_h^{n+1},\mb{u}_h^{n+1}, p_h^{n+1}) \bigr)_{\Omega} \\
		&  \geq \frac{3}{2\Dt} \norm{\rho_h^{n+1}}^2_{0,\Omega}-\frac{5}{2\Dt}C_{\rho}\norm{\rho_h^{n+1}}_{0,\Omega}
		\\& \tilde{\alpha}_a\norm{\mb{u}_h^{n+1}}^2_{0,\Omega}  - \norm{\mb{f}^{n+1}}_{0,\Omega}\norm{\mb{u}_h^{n+1}}_{0,\Omega}.
	\end{align*}
	Then, setting 
	\begin{align*} 
		C_R = \min\left \{ \frac{3}{2\Dt}, \tilde{\alpha}_a \right\}\mbox{  and  }  C_r = \sqrt{2} \max \left\{ \frac{5}{2\Dt} C_{\rho}, \norm{\mb{f}}_{L^{\infty}(0,T,\mb{L}^2(\Omega))} \right\},
	\end{align*}  
	we may apply the inequality $a+b \leq \sqrt{2} (a^2+b^2)^{1/2}$, valid for all $a,b\in \mathbb{R}$, to obtain
	\begin{align*}
		&	\bigl( \Phi(\rho_h^{n+1},\mb{u}_h^{n+1}, p_h^{n+1}),(\rho_h^{n+1},\mb{u}_h^{n+1}, p_h^{n+1}) \bigr)_{\Omega} \\ &\quad \geq  
		C_{R} \bigl(\norm{\rho_h^{n+1}}_{0,\Omega}^2 +  \norm{\mb{u}_h^{n+1}}_{0,\Omega}^2 \bigr)   
		- C_{r} \bigl(\norm{\rho_h^{n+1}}_{0,\Omega}^2 + \norm{\mb{u}_h^{n+1}}_{0,\Omega}^2 \bigr)^{1/2}.
	\end{align*} 
	Hence, the right-hand side is nonnegative on a sphere of radius $r  \coloneqq   C_r / C_R$. Consequently, by  
	Theorem~\ref{th:brouwer}, there exists a solution to the fixed-point problem 
	$\Phi  (\rho_h^{n+1},\mb{u}_h^{n+1}, p_h^{n+1})=0$, {where the fixed-point map \eqref{eq:fixed} is the solution operator for the fully discrete problem \eqref{eq:full_disc_scheme}.} 
\end{proof}

Note that, even when uniqueness of the discrete scheme remains an open
problem, our non-exhaustive selection of numerical examples did not present any
difficulties in this regard.

\subsection{Semismooth Newton Linearization and multiplier approach}

At each time iteration, we are left with a nonlinear system, which involves the non-differentiable function associated with the Huber regularization $|\cdot|_\gamma$. This fact prevents us from proposing a Newton iteration to solve such a system.  Despite this drawback, our goal remains to have a fast-converging method to solve this system. Thus, we propose a semismooth Newton (SSN) iteration, which uses either Newton or slantly differentiation. For the sake of readability of the paper, we provide the definition of slantly differentiation.
\begin{defi}
Let $X$ and $Y$ be two Banach spaces, and let $D\subset X$ be an open domain. A function $F:D\subset X\rightarrow Y$ is said to be slantly differentiable at $x\in D$ if there exists a mapping $G_F:D\rightarrow\mathcal{L}(X,Y)$ such that the family $\{G_F(x+h)\}$ of bounded linear operators is uniformly bounded in the operator norm for $h$ sufficiently small and
\[
\underset{h\rightarrow 0}{\lim}\frac{F(x+h)-F(x)- G_F(x+h) h}{\|h\|}=0.
\]
\end{defi}
The use of this differentiation concept is justified since it is well known that both the max function and the Frobenius norm are slantly differentiable in finite-dimensional spaces (see \cite{JC2012,Gon1,GonMen} and references therein). Furthermore, the SSN approach has been shown to be efficient and provides a linearization scheme that exhibits superlinear convergence when applied to discretized viscoplastic models, as discussed in the aforementioned literature. 

We also introduce a multiplier approach with an auxiliary tensor $\mb{z}$ such that
$|\mb{D}\mb{u}|_{\gamma} \mb{z}=\gamma \tau_s \mb{D}\mb{u}$. The strategy is a particularly efficient numerical technique for solving viscoplastic flow problems in which the nonlinearity is related to the unknown velocity gradient. Moreover, the new formulation is equivalent to the original problem in the continuous case (this can be proven using the same techniques as in \cite[Proposition 3.7]{GonMen}). In the discrete case, we take $\mb{z}_h \in  \mb{W}_{h}$, where
\[\mb{W}_h \coloneqq \left\{ \mb{w} \in \mathbb{L}^{2}(\Omega) \,:\, \mb{w}|_{K} \in [\mathbb{P}^{k-1}(K)]^{d\times d}, \quad \forall K \in \mathcal{T}_h \right\},\]
 and add the following equation:
\begin{equation} \label{eq:tensor_z}
(\gamma \tau_s \mathbf{D}_h \mb{u}_h^{n+1}, \mb{w}_h)_{\Omega} - (\abs{\mathbf{D}_h \mb{u}_h^{n+1}}_{\gamma} \mb{z}_h^{n+1},\mb{w}_h)_{\Omega} =0.\end{equation}
	
	Note that if we take $\mb{w}_h = \frac{\mathbf{D}_h \mb{u}_h^{n+1}}{|\mathbf{D}_h \mb{u}_h^{n+1}|_{\gamma}}$ (which is possible in the case $k=1$ that we use for our numerical tests, since $\mb{u}_h^{n+1} \in \mb{H}(\Div;\Omega)$ and $\mathbf{D}_h \mb{u}_h^{n+1}|_K \in [\mathbb{P}^{0}(K)]^{d\times d}$), we can deduce the inequality,
	\begin{align*}
		0 \leq \tau_s \gamma \int_{\Omega}{\frac{\abs{\mathbf{D}_h \mb{u}_h^{n+1}}^2}{\abs{\mathbf{D}_h \mb{u}_h^{n+1}}_{\gamma}}}\dx = (\mb{z}_h^{n+1},\mathbf{D}_h \mb{u}_h^{n+1})_{\Omega}.
	\end{align*}
	which in turn allow us to maintain our stability and existence results. The additional tensor is particular useful to improve the SSN convergence for large Reynolds number simulations.
	
Given the discussion above,  the semismooth Newton linearization for system  \eqref{eq:full_disc_scheme}, including \eqref{eq:tensor_z}, about $(\rho_h^{n+1},\mb{u}^{n+1}_h,p_h^{n+1},\mb{z}^{n+1}_h)$, gives the following problem: find $\delta_{\rho} \in \mathcal{W}_h$, $\delta_{\mb{u}}\in \mb{V}_h$, $\delta_p\in \mathcal{Q}_h$, $\delta_{\mb{z}} \in \mb{W}_h$, such that, for all $\zeta_h \in \mathcal{W}_h$, $\mb{v}_h \in \mb{V}_h$, $q_h\in \mathcal{Q}_h$ and $\mb{w}_h\in\mb{W}_h$, it holds that

\begin{subequations}\label{SSN}
\begin{equation}\label{SSN1}
	\begin{array}{lll}
\frac{3}{2\Dt} \int_\Omega \delta_{\rho}\zeta_h\dx + c_{1,\delta}^h(\mb{u}_h^{n+1},\delta_{\mb{u}}, \rho_h^{n+1},\delta_{\rho}, \zeta_h)=-\frac{1}{2 \Dt} \int_\Omega\bigl(\mathcal{D} \rho_h^{n+1} \zeta_h \bigr)\dx - c_1^h(\mb{u}_h^{n+1}, \rho_h^{n+1}, \zeta_h) 
	\end{array}\vspace{0.2cm}
\end{equation}

\begin{equation}\label{SSN2}
	\begin{array}{lll}
	\frac{3}{2\Dt} \int_\Omega\bigl( (\rho_h^{n+1} \delta_{\mb{u}} + \delta_\rho \mb{u}_h^{n+1} \bigr)\cdot\mb{v}_h)\dx + \tilde{a}^h_{2}(\delta_{\mb{z}};\delta_{\mb{u}},\mb{v}_h) + b(\delta_p,\mb{v}_h)\vspace{0.2cm} \\
	\hspace{0.5cm}-\int_\Omega\frac{\tau_s\chi_{\mathcal{A}_\gamma}}{\abs{\djump{\mathbf{D}_h\mb{u}^{n+1}_h}}^3} \left(\djump{\mathbf{D}_h\mb{u}^{n+1}_h}\,:\,\djump{\mathbf{D}_h\delta_{\mb{u}}}\right)\,(\dmean{\mb{D}_h\mb{v}_h}\mb{n}_e\,:\,\djump{\mb{D}_h\mathbf{u}_h})\,dx \vspace{0.2cm}
	\\\hspace{0.5cm}
+ c^h_{2,\delta}(\rho_h^{n+1},\delta{\rho},\mb{u}_h^{n+1},\delta_{\mb{u}};\mb{u}_h^{n+1},\delta_{\mb{u}},\mb{v}_h)  
=-\frac{1}{2\Dt} \int_\Omega\bigl( \sigma_h^{n+1} \mathcal{D} (\sigma_h^{n+1} \mb{u}_h^{n+1}\bigr)\cdot\mb{v}_h)\dx\vspace{0.2cm} \\\hspace{3.5cm} - \tilde{a}^h_2(\mb{z}_h^{n+1};\mb{u}_h^{n+1},\mb{v}_h) - c^h_2(\rho_h^{n+1}\mb{u}_h^{n+1};\mb{u}_h^{n+1},\mb{v}_h)  \vspace{0.2cm} \\\hspace{5.5cm}- b(p_h^{n+1},\mb{v}_h)+\int_\Omega \mb{f}^{n+1}\cdot\mb{v}_h\dx,
	\end{array}\vspace{0.3cm}
\end{equation}

\begin{equation}\label{SSN3}
b(q_h,\delta_{\mb{u}})=-b(q_h,\mb{u}_h^{n+1}),
\end{equation}

\begin{equation}\label{SSN4}
	\begin{array}{lll}
		\gamma\tau_s\int_\Omega \mathbf{D}_h\delta_{\mb{u}}\,:\mathbf{w}_h\,dx -\gamma\int_\Omega\frac{\chi_{\mathcal{A}_\gamma}}{\abs{\mathbf{D}_h\mb{u}^{n+1}_h}} \left(\mathbf{D}_h\mb{u}^{n+1}_h\,:\,\mathbf{D}_h\delta_{\mb{u}}\right)\,(\mathbf{z}^{n+1}_h\,:\,\mathbf{w}_h)\,dx \vspace{0.2cm}\\\hspace{0cm}  -\int_\Omega |\mathbf{D}_h\mb{u}^{n+1}_h|_\gamma\,\delta_{\mathbf{z}}\,:\,\mathbf{w}_h\,dx= -\gamma\tau_s \int_\Omega \mathbf{D}_h\mb{u}^{n+1}_h\,:\,\mathbf{w}_h\,dx + \int_\Omega \abs{\mathbf{D}_h\mb{u}^{n+1}_h}_\gamma \mathbf{z}^{n+1}_h\,:\,\mathbf{w}_h\,dx,
	\end{array}\vspace{0.2cm}
\end{equation}
\end{subequations}

where 
\begin{equation}
\begin{split}
\tilde{a}^h_2(\mb{z}_h;\mb{u}_h, \mb{v}_h)  
	& \coloneqq \int_{\Omega} \bigl(\nu\mathbf{D}_h\mb{u}_h+\mb{z}_h:\mathbf{D}_h\mb{v}_h \bigr)\dx   \label{eq:a2hdef2}\\
	&  \quad + \sum_{e\in\mathcal{E}_h} \int_e \biggl( -\dmean{(\nu\mathbf{D}_h(\mb{u}_h)+\mb{z}_h)\mb{n}_e} \cdot \djump{\mb{v}_h }  \\
	& - \dmean{\mu(\abs{\djump{\mathbf{D}_h \mb{u}_h}}_{\gamma}) \mathbf{D}_h(\mb{v}_h)\mb{n}_e}\cdot\djump{\mb{u}_h } + \frac{a_0}{h_e} \djump{\mb{u}_h } : \djump{\mb{v}_h } \biggr) \ds.
\end{split}
\end{equation}

Let us discuss the equation \eqref{SSN4}, associated with the Huber term $|\cdot|_\gamma$. Here, we have that
\[
\chi_{\mathcal{A}_\gamma}:=\left\{
\begin{array}{lll}
	1,& \mbox{if $\abs{\mathbf{D}_h \mb{u}_h}\geq \frac{\tau_s}{\gamma}$}\vspace{0.2cm}\\0,&\mbox{otherwise}.
\end{array}
\right. 
\]
This function stands for the slantly derivative of the $\max$ term in $\abs{\cdot}_\gamma$, and gives us a good estimator of the approximated yielded and unyielded regions in the material, respectively \cite{JC2012}.  The regions in which $\chi_{\mathcal{A}_\gamma}=1$ are the active sets in the smoothing step and corresponds to the Huber approximations of the yielded regions. Respectively, the regions where $\chi_{\mathcal{A}_\gamma}=0$, are the inactive sets and correspond to the Huber approximations of the unyielded regions. 

Next, let us focus on the convective forms $c_1^h(\mathbf{u}_h, \rho_h,\zeta_h)$ and $c_2^h(\rho_h\mathbf{w}_h,\mathbf{u}_h,\mathbf{v}_h)$. These forms are well posed due to the dG formulation and the analysis done in Section \ref{sec:stab}.  In consequence,  they are differentiable with derivatives given by  $c_{1,\delta}^h(\mathbf{u}_h, \rho_h,\zeta_h)$ and $c_{2,\delta}^h(\rho_h\mathbf{w}_h,\mathbf{u}_h,\mathbf{v}_h)$

\begin{equation*}
\begin{array}{lll}	
	c^h_{1,\delta} (\mb{u}_h,\delta_{\mb{u}}; \rho_h,\delta_{\rho}, \zeta_h) &:=& \int_\Omega (\mb{u}_h \cdot \nabla \delta_{\rho} +\delta_{\mb{u}}\cdot \nabla \rho_h)\zeta_h \dx \vspace{0.2cm}\\&&-\sum_{e\in\mathcal{E}_h} \int_e ((\delta_{\mb{u}}\cdot \mb{n}_e) \djump{\rho_h}+(\mb{u}_h\cdot \mb{n}_e) \djump{\delta_\rho})\dmean{\zeta_h}\ds\vspace{0.2cm} \\&&+ \sum_{e\in\mathcal{E}_h} \int_e \frac{1}{2}\frac{\mb{u}_h\cdot \mb{n}_e}{\abs{\mb{u}_h\cdot \mb{n}_e}} (\delta_{\mb{u}}\cdot\mb{n}_e)\djump{\rho_h}\cdot\djump{\zeta_h}\ds\vspace{0.2cm}\\&&+ \sum_{e\in\mathcal{E}_h} \int_e \frac{1}{2}\abs{\mb{u}_h\cdot \mb{n}_e} \djump{\delta_{\rho}}\cdot\djump{\zeta_h}\ds
\end{array}
\end{equation*}
and
\begin{equation*}
\begin{array}{lll}	
c^h_{2,\delta} (\rho_h,\delta_{\rho}, \mb{w}_h,\delta_{\mb{w}}; \mb{u}_h,\delta_{\mb{u}}, \mb{v}_h) := \int_\Omega ((\rho_h \delta_{\mb{w}}+\delta_{\rho} \mb{w}_h) \cdot \nabla_h)\mb{u}_h\cdot\mb{v}_h \dx + \int_\Omega (\rho_h \mb{w}_h \cdot \nabla_h)\delta_{\mb{u}}\cdot\mb{v}_h \dx\vspace{0.2cm} \\\hspace{2cm}
 +\frac{1}{2}\int_{\Omega}(\nabla \cdot(\delta_{\rho} \mb{w}_h+\rho_h \delta_{\mb{w}})\mb{u}_h + \nabla \cdot(\rho_h \mb{w}_h)\delta_{\mb{u}})\cdot \mb{v}_h \dx\vspace{0.2cm} \\\hspace{3cm} - \sum_{e\in\mathcal{E}_h} \int_e ((\delta_{\rho}\mb{w}_h+\rho_h\delta_{\mb{w}})\cdot \mb{n}_e) \djump{\mb{u}_h}+(\rho_h\mb{w}_h\cdot \mb{n}_e) \djump{\delta_{\mb{u}}})\cdot\dmean{\mb{v}_h}\ds \vspace{0.2cm} \\\hspace{3.5cm}
	+ \sum_{e\in\mathcal{E}_h} \int_e \left(\frac{1}{2}\abs{\rho_h\mb{w}_h\cdot \mb{n}_e} \djump{\delta_{\mb{u}}}+\frac{1}{2}\frac{\rho_h\mb{w}_h\cdot \mb{n}_e}{\abs{\rho_h\mb{w}_h\cdot \mb{n}_e}}(\delta_{\rho}\mb{w}_h+\rho_h\delta_{\mb{w}})\cdot\mb{n}_e \djump{\mb{u}_h}\right)\cdot\djump{\mb{v}_h}\ds.
\end{array}
\end{equation*}

Summarizing, we can conclude that system \eqref{SSN} is well-posed (see \cite{GonMen} for further details). Moreover, by following a similar analysis as the one in \cite{JC2012}, we can state that the SSN iteration converges superlinearly locally. This assertion will be computationally confirmed in the numerical experiments carried out in the next section. 

\section{Numerical results} \label{sec:Num-examples}

In this section, we test the performance of the numerical method on a set of quasi-uniform triangulations of the respective domain. The implementation of the $\mathbf{H}(\mathrm{div})$-conforming finite element scheme is carried out using the open source finite element library FEniCS \cite{alnaes2015} and polynomial degree $k=1$. The linear systems encountered at each Semismooth Newton step are solved with the multifrontal massively parallel sparse direct solver MUMPS. The Newton iterations terminate when either the absolute or the relative residuals (measured in the $\ell_2$-norm) fall below a fixed tolerance of $1\times10^{-5}$. 

\subsection{Constant Density}

 \begin{figure}[!t]
	\begin{center}
		\includegraphics[height=0.4\textwidth]{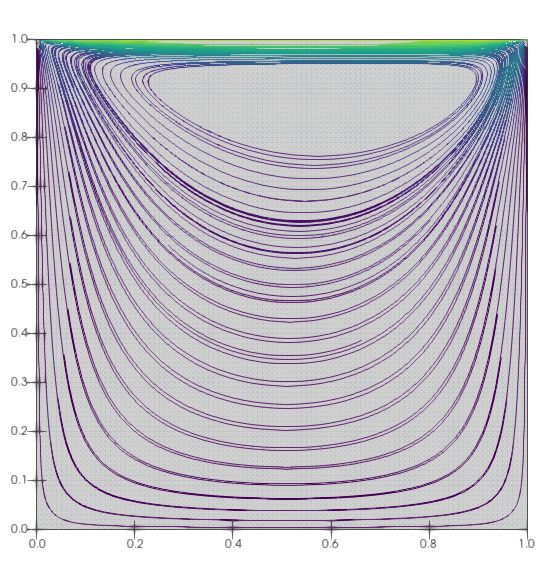}\includegraphics[height=0.4\textwidth]{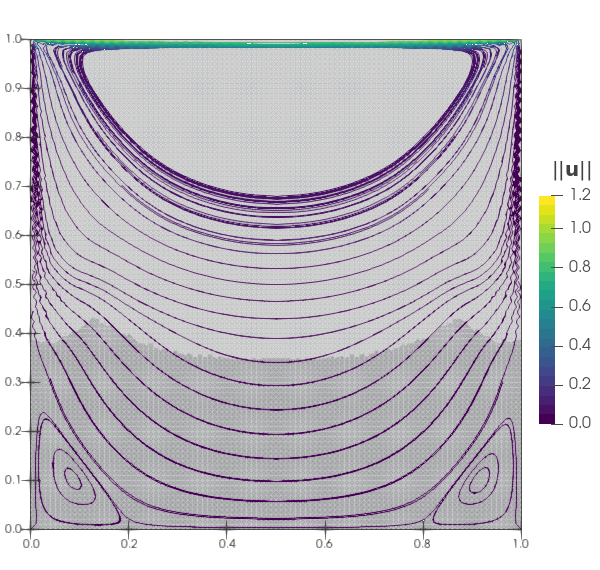}\\
		\includegraphics[height=0.4\textwidth]{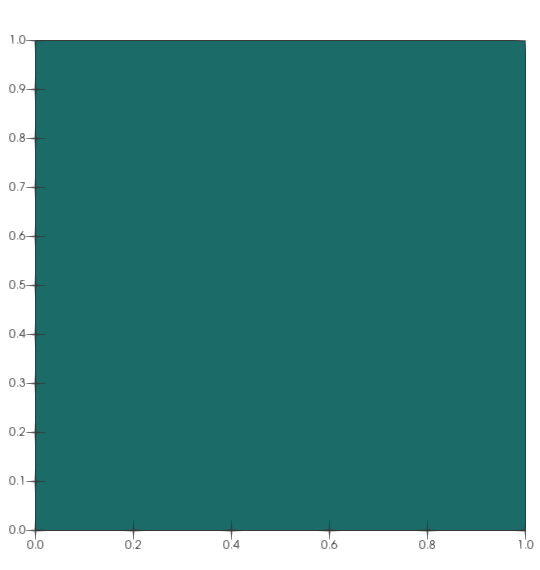}\includegraphics[height=0.4\textwidth]{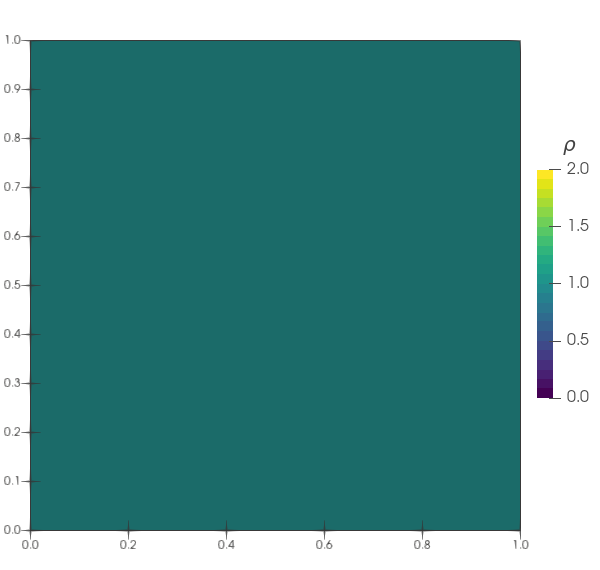}
	\end{center}
	\vspace{-3mm}
	\caption{lid-driven cavity: Velocity vector field stream lines/inactive sets ($\mathcal{I}_{\gamma}$, dark gray) and constant density for $\tau_s = 0.0 $ (left) and $\tau_s = 2.5$ (right), at time $t=0.5$. Parameters: $\mathrm{Re}=100$}\label{fig:ex0-1}  
\end{figure}

We start by testing a standard two-dimensional lid-driven cavity with constant density at Re=100, to verify that our method does not introduce spurious variations in density. As shown in Figure \ref{fig:ex0-1}, using a $100\times 100$ mesh, the method preserves constant density, and the velocity streamlines, as well as the active/inactive zones, are in good agreement with similar examples computed using other numerical schemes (see i.e \cite{Botti2021}, \cite{GonMen}, \cite{JC2012}, \cite{Narain2022}).

 \begin{figure}[!t]
	\begin{center}
		\includegraphics[height=0.4\textwidth]{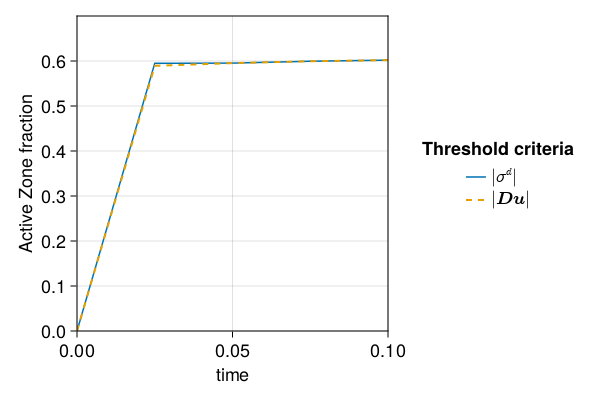}
	\end{center}
	\vspace{-3mm}
	\caption{lid-driven cavity: Fraction of cells corresponding to the active zone at different times, computed with different criteria for the threshold. Parameters: $\tau_s = 2.5$, $\gamma=\num{1e3}$, $\mathrm{Re}=100$.}\label{fig:th_crit}  
\end{figure}

Usually, there are two main criteria for approximating the yielding and unyielding zones numerically. One is to compute a threshold in terms of the norm of the deviatoric part of the stress tensor ($\bs{\sigma}^d$), known as the von Mises criterion. The other is to define the threshold in terms of the magnitude of the shear rate ($\mb{Du}$), as proposed in \cite{Saramito}. As depicted in Figure \ref{fig:th_crit}, the two criteria are highly consistent in our scheme, with only a small difference in the first Euler time iteration. Therefore, we define the active/inactive zones based on the magnitude of the shear rate and calculated using the approximation shown in equation \eqref{eq:zones_approx} for the remainder of this section.

 \begin{figure}[!t]
	\begin{center}
		\includegraphics[height=0.4\textwidth]{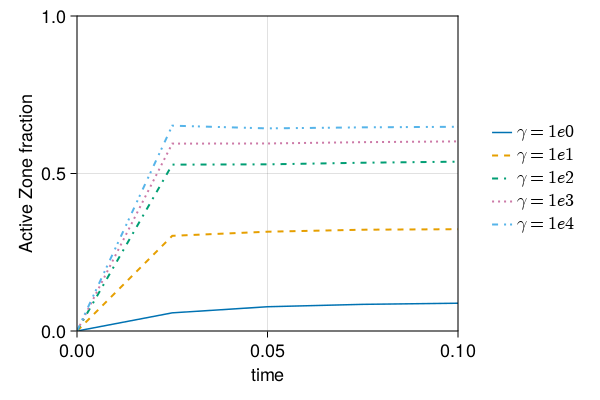}
	\end{center}
	\vspace{-3mm}
	\caption{lid-driven cavity: Fraction of cells corresponding to the active zone computed with different $\gamma$ values. Parameters: $\tau_s = 2.5$, $\mathrm{Re}=100$.}\label{fig:az_gamma}  
\end{figure}

We also test the impact of the regularization parameter $\gamma$ on the size of the active/inactive zone. In Figure \ref{fig:az_gamma}, we present the fraction of cells corresponding to the active zone at different times and using different values for $\gamma$. Note that the zone size largely changes for values below $\num{1e2}$. Theoretically, the approximation improves as $\gamma \to \infty$. However, the condition of the resultant matrix deteriorates as $\gamma$ increases. From now on, we fix $\gamma=\num{1e3}$ as a trade-off between these competing criteria.

\subsection{Analytical Solution}

	Only a few analytical solutions are available for viscoplastic fluid problems. One such solution is reported for the stationary Bingham fluid problem with constant density and velocity field $\mb{u} = (u_1,0)$. In two dimensions is given by
	\begin{align} \label{eq:an_sol}
		u_1 = \begin{cases}
			\frac{1}{8}[(1-2\tau_s^2)-(1-2\tau_s-2y)^2], &\text{if }0\leq y <\frac{1}{2}-\tau_s \\
			\frac{1}{8}(1-2\tau_s)^2, & \text{if }\frac{1}{2}-\tau_s \leq y \leq \frac{1}{2}+\tau_s \\
			\frac{1}{8}[(1-2\tau_s)^2-(2y-2\tau_s-1)^2], &\text{if }\frac{1}{2}+\tau_s < y \leq 1
		\end{cases}
	\end{align}
which corresponds to the flow between two parallel plates.

We use this simplified setting to test the ability to recover analytical solutions and check the convergence rates of the $\mathbf{H}(\mathrm{div})$-conforming discretization for the Bingham fluid problem. We consider $\Omega = ]0,1[\times]0,1[$, $\eta =1.0$, $\mb{f}=\mb{0}$, and Dirichlet boundary conditions are imposed on the domain according to \eqref{eq:an_sol}. Table \ref{tab:ex1_err} shows the numerical error in the discrete norms
\begin{align*}
\norm{\mb{u}}_{0,\Th} \coloneqq  \left( \sum_{n=1}^N \norm{\mb{u}_h^n}_{1,\mathcal{T}_h}^2\right)^{1 / 2}, \quad \text{and} \quad 
\norm{p}_{0,k} \coloneqq  \left(\sum_{n=1}^N \norm{p_h^n}_{k,\Omega}^2\right)^{1/ 2}. 
\end{align*}
The corresponding individual errors and convergence rates are computed as 
\begin{gather}
\texttt{e}_{\mb{u}} = \norm{\mb{u} - \mb{u}_h}_{0,\Th}, \quad \texttt{e}_{p} = \norm{p - p_h}_{0,0}, \nonumber\\ 
\texttt{rate} =\log(e_{(\cdot)}/\tilde{e}_{(\cdot)})[\log(h/\tilde{h})]^{-1}, \label{eq:error01}
\end{gather}
where $e,\tilde{e}$ denote errors generated on two
consecutive pairs of mesh size ~$h$, and~$\tilde{h}$, respectively.

\begin{table}[t]
	\setlength{\tabcolsep}{4pt}
	\renewcommand{\arraystretch}{1.3}
	\centering 
	{\small\begin{tabular}{|r|ccccc|}
		\hline 
		h & $\texttt{e}_{\mb{u}}$ & \texttt{rate} & $\texttt{e}_{p}$ & \texttt{rate} & $\norm{\mathrm{div}\mb{u}_h}_{\infty,\Omega}$\tabularnewline
		\hline
		\hline
		 0.5&1.0142&  --- &    8.5903 & --- & 2.7756e-17 \tabularnewline
		 0.25&0.4797 & 1.0802 & 3.9553 & 1.1189 &8.3267e-17 \tabularnewline
		 0.125&0.0979& 2.2929 & 0.9309 & 2.0871 &2.2205e-16 \tabularnewline
		 0.0625&0.0285& 1.7800 & 0.2588 & 1.8466&5.5511e-16 \tabularnewline
		 0.03125&0.0071& 1.9972& 0.0909& 1.5099&8.8818e-16 \tabularnewline
		\hline 
	\end{tabular} }
	\medskip\caption{Experimental errors and convergence rates for the approximate solutions 
		$\mb{u}_h$, $p_h$, {where the polynomial degree $k=1$ is used}. The $\ell^{\infty}$-norm of the vector formed 	by the divergence of the discrete velocity computed for each discretization is shown in the last column.} \label{tab:ex1_err}
\end{table}

Notice that the convergence rates are higher than what is theoretically expected for Navier-Stokes type problems (see, e.g., \cite{Guzman2016, konno2011}), but they are close to the expected rates ($\mathcal{O}(h^{k+1})$) for the Darcy equation (as reported in \cite{konno2011}). Furthermore, we observe that the total error is dominated by the pressure approximation and that the discrete velocities are indeed divergence-free.

\subsection{Viscous Rayleigh-Taylor Instability}

As a test case, we consider the physically interesting problem of the development of Rayleigh-Taylor instability in the viscous regime. This problem has been studied in previous works such as \cite{Freignaud2001,Calgaro2008}, which build upon the work by Tryggvason \cite{Tryggvason1988}. We consider a domain $\Omega = ]-l/2,l/2[\times]-2l,2l[$ filled with two layers of fluid with varying density, initially at rest and subject to gravity. Note that, to allow for comparison with previous studies on the Navier-Stokes setting, we only consider differences in the fluids' density, and no other property. Thus, the yield stress is assumed to be the same for both fluids. As proposed in \cite{Tryggvason1988}, the interface at time $t = 0$, is given as follows:
\begin{align*}
	\rho_0(x,y) = \frac{\rho_m+\rho_M}{2} + \frac{\rho_m-\rho_M}{2} \tanh \left(\frac{y-\omega \cos(2\pi x/l)}{0.01l}\right),
\end{align*}
where $\rho_M>\rho_m>0$ and $\omega > 0$ is the amplitude of the initial perturbation.

Set in $l$ the representative column length; we define dimensionless variables as
\begin{align*}
	\tilde{\rho} = \frac{\rho}{\rho_m}, \quad \tilde{\mb{x}} = \frac{\mb{x}}{l}, \quad \tilde{t} = \frac{t}{l^{1/2} g^{-1/2}},\quad \tilde{\mb{u}} = \frac{\mb{u}}{l^{1/2}g^{1/2}},	
\end{align*}
and we also define the following dimensionless numbers: The density ratio is measured by the Atwood number,
\begin{align*}
	\mathrm{At} = \frac{\rho_M-\rho_m}{\rho_M+\rho_m},
\end{align*}
and the the Reynolds number is defined as
\begin{align*}
	\mathrm{Re} = \frac{\rho_m l^{3/2} g^{1/2}}{\eta},
\end{align*}
where $\eta > 0$ is the dynamic viscosity of the fluid and $g$ is the gravitational acceleration (when $\tau_s >0$, the maximum Reynolds number is reported). We take $\mb{f} = \rho \mb{g}$, with $\mb{g} = (0,-g)$. Furthermore, when presenting our numerical results for this example, we will use the time scale of Tryggvason (we have $t_{\mathrm{Try}} = \tilde{t} \sqrt{\mathrm{At}}$).

Making abuse of notation, in what follows, we will write it simply $\rho$, $\mb{x}$, $t$ and $\mb{u}$, instead of $\tilde{\rho}, \tilde{\mb{x}}, \tilde{t}$ and $\tilde{\mb{u}}$, respectively, when no confusion can arise.

We compute the solution on the domain $(-l/2, l/2) \times (-2l, 2l)$ with the following boundary conditions for the velocity field: $\mb{u} = \mb{0}$ on the horizontal boundaries, and $\mb{u} = (0,v)$ with $\nabla \mb{u}\cdot\mb{n} = 0$ on the vertical boundaries.

 \begin{figure}[!t]
	\begin{center}
		\includegraphics[height=0.6\textwidth]{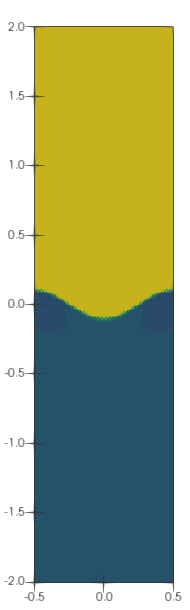}\includegraphics[height=0.6\textwidth]{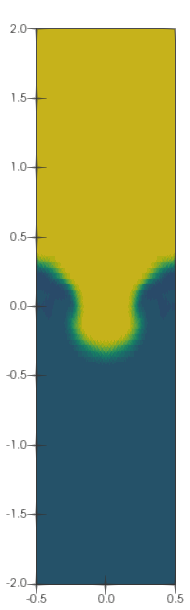}\includegraphics[height=0.6\textwidth]{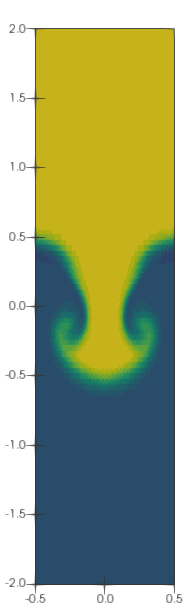}\includegraphics[height=0.6\textwidth]{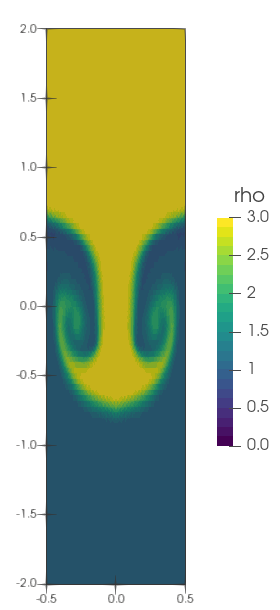}\\ \includegraphics[height=0.6\textwidth]{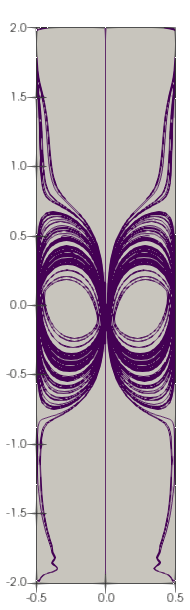}\includegraphics[height=0.6\textwidth]{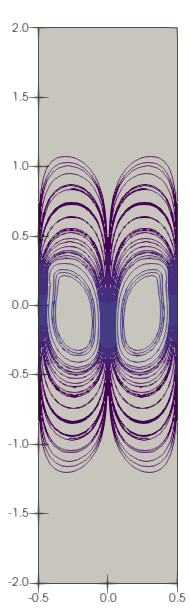}\includegraphics[height=0.6\textwidth]{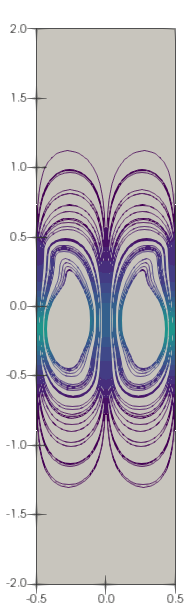}\includegraphics[height=0.6\textwidth]{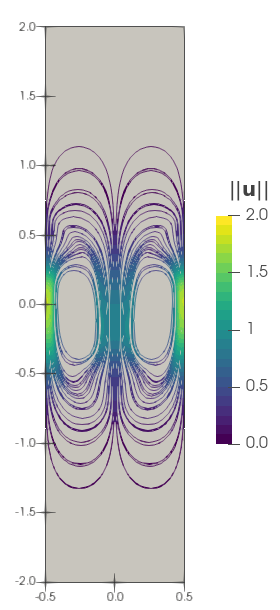}
	\end{center}
	\vspace{-3mm}
	\caption{Rayleigh-Taylor instability. Top: Evolution of the density interface. Bottom: Velocity vector field stream lines and inactive set (dark gray). Times: 0.1, 1.0, 1.5 and 2.0 (from left to right). Parameters: $\tau_s = 0.0$, $\mathrm{Re}=1000$, $\omega = 0.1$}\label{fig:ex1-1}  
\end{figure}

As a sanity check, we start by analyzing the case with a zero plasticity threshold, in order to compare our qualitative results with previous works on Navier-Stokes variable density incompressible flows. We set $\mathrm{At} = 0.5$ (i.e., $\rho_M = 3$, $\rho_m = 1$) and an initial condition of $\omega = 0.1$. We simulate a low Reynolds case with $\mathrm{Re} = 1000$ using a $100\times 100$ cell mesh. Comparing our qualitative results displayed for different time snapshots in Figure \ref{fig:ex1-1} with those presented in \cite[Figure 4]{Bell92}, \cite[Figure 11]{Calgaro2008}, and \cite[Figure 1]{Freignaud2001}, there is good agreement of the density profile in the early stages, with only some differences observed at large times. As noted in \cite{Freignaud2001}, these differences can be expected since an accurate and detailed prediction of the flow is usually difficult for $t \geq 1.5$.

 \begin{figure}[!t]
	\begin{center}
		\hspace{0.6cm}\includegraphics[height=0.6\textwidth]{RTEx1ut1_5.png}\includegraphics[height=0.6\textwidth]{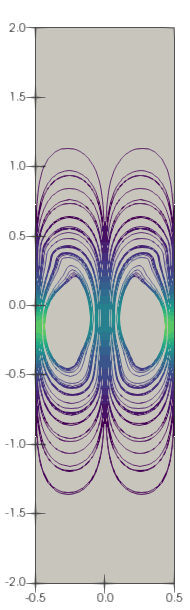}\includegraphics[height=0.6\textwidth]{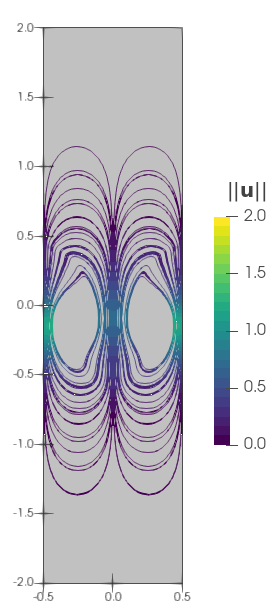}\\
		\hspace{7mm}\includegraphics[height=0.604\textwidth]{RTEx1rhot1_5.png}\includegraphics[height=0.6\textwidth]{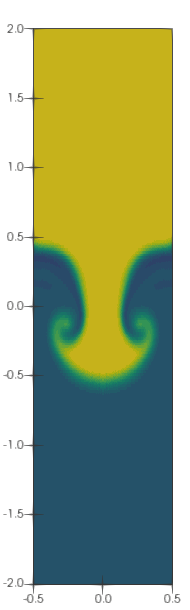}\includegraphics[height=0.6\textwidth]{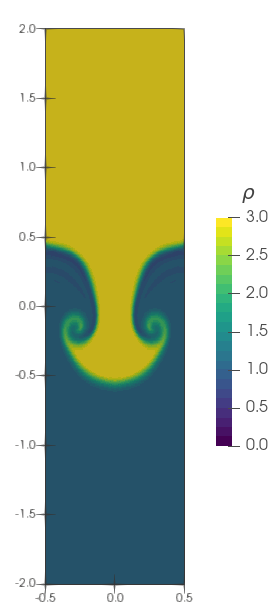}
	\end{center}
	\vspace{-3mm}
	\caption{Rayleigh-Taylor instability. Velocity stream lines and density interface at time t = 1.5 for 40 000, 160 000 and 640 000 cells. Parameters: $\tau_s=0.0$, $\mathrm{Re}=1000$, $\omega = 0.1$}\label{fig:ex1-1b}  
\end{figure}

In Figure \ref{fig:ex1-1b}, we also test the same setting with different mesh sizes: $100\times 100$ (40 000 cells), $200\times 200$ (160 000 cells), and $400\times 400$ (640 000 cells). The solutions largely agree between them, and the main features of the solution are still present in the coarse mesh. However, the details of counter-rotating swirls continue to improve with each refinement.

 \begin{figure}[!t]
	\begin{center}
	\includegraphics[height=0.5\textwidth]{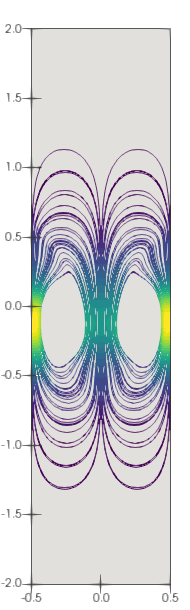}\includegraphics[height=0.5\textwidth]{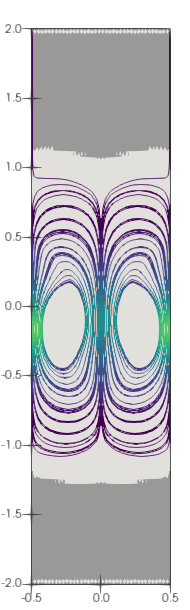}\includegraphics[height=0.5\textwidth]{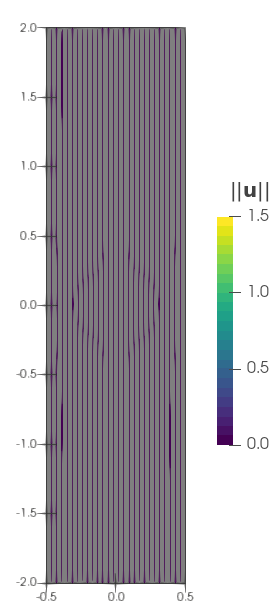}
	\includegraphics[height=0.5\textwidth]{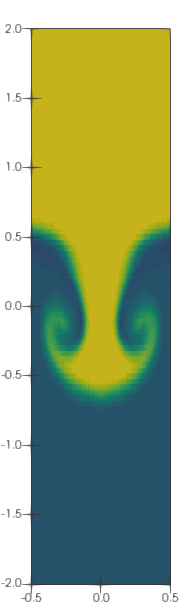}\includegraphics[height=0.5\textwidth]{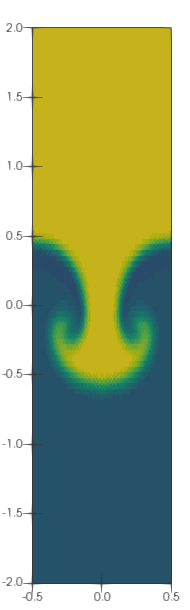}\includegraphics[height=0.5\textwidth]{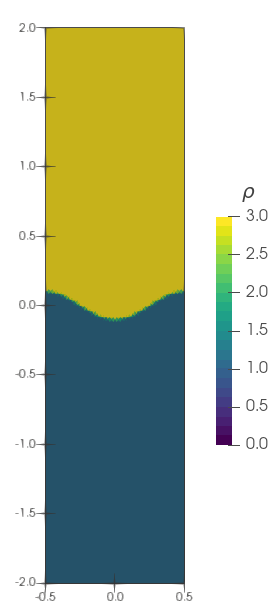}
	\end{center}
	\vspace{-3mm}
	\caption{Rayleigh-Taylor instability. Velocity stream lines and inactive set ($\mathcal{I}_\gamma$, dark gray); and density interface at time t = 1.75 for $\tau_s=0.0$, $\tau_s=\num{1e-1}$ and $\tau_s=\num{1.0}$. Parameters: $\mathrm{Re}=1000$, $\omega = 0.1$}\label{fig:ex1-2}  
\end{figure}

The influence of the yield stress $\tau_s$ is displayed in Figure \ref{fig:ex1-2}, where ascending counter-rotating vortices develop more slowly as the yield stress increases. As expected, the active set ($\mathcal{A}_{\gamma}$) also decreases. In fact, for the final test with a value of $\tau_s=1.0$, there is no vortex development at time $t=1.75$, with an almost imperceptible change in density interfaces.


 \begin{figure}[!t]
	\begin{center}
		\includegraphics[height=0.5\textwidth]{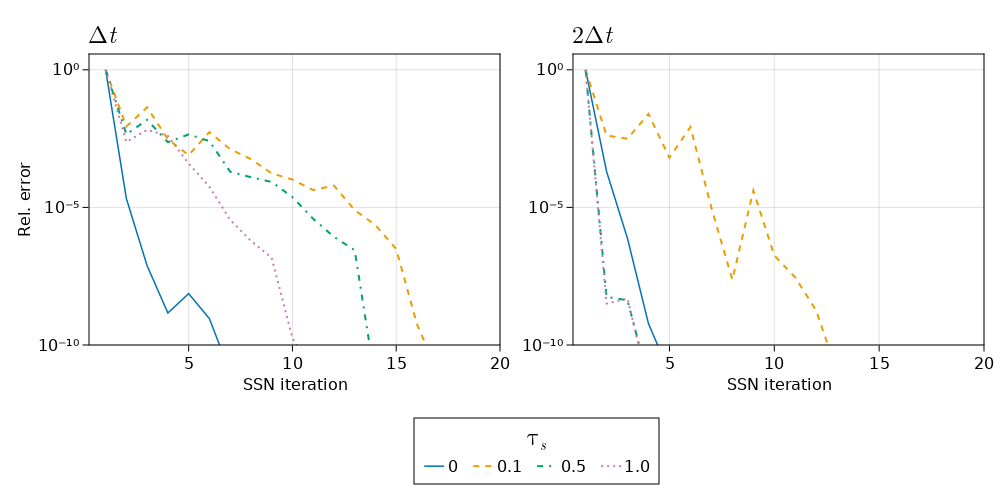}
	\end{center}
	\vspace{-3mm}
	\caption{Rayleigh-Taylor instability: Relative Error vs SSN iterations for the first two time iterations, with $\tau_s = 0.0, 0.1, 0.5$ and $1.0$. Parameters: $\omega = 0.1$, $\Delta t =0.1$}\label{fig:ex1-conv}  
\end{figure}

The relative error for each SSN iteration is displayed in Figure \ref{fig:ex1-conv} for the $\mathrm{At} = 0.5$ setting, $\tau_s = 0, 0.1, 0.5,$ and $1.0$ for the first two time iterations. As can be seen, convergence is slower when $\tau_s$ is close to $0.1$. In all cases, the second iteration converges faster. In general, fewer Newton iterations are required as the initial approximation improves across time iterations.


 \begin{figure}[!t]
	\begin{center}
\includegraphics[height=0.55\textwidth]{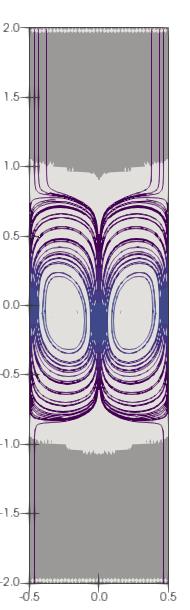}\includegraphics[height=0.55\textwidth]{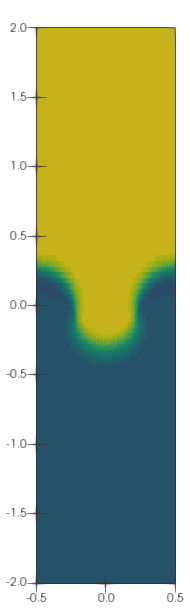}\hspace{5mm}\includegraphics[height=0.55\textwidth]{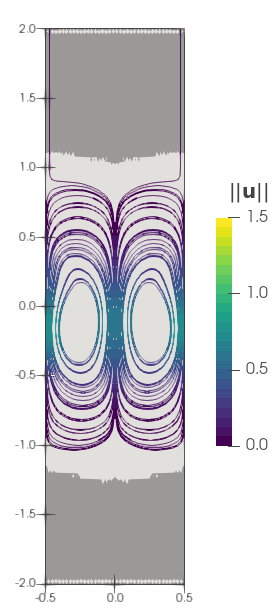}\includegraphics[height=0.55\textwidth]{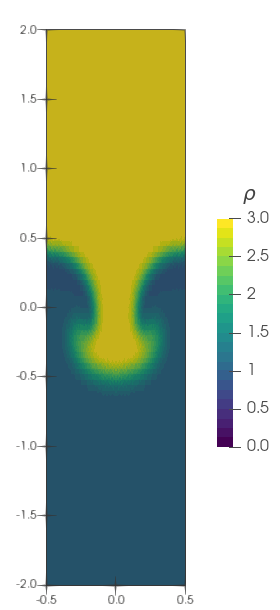} \\
	\includegraphics[height=0.55\textwidth]{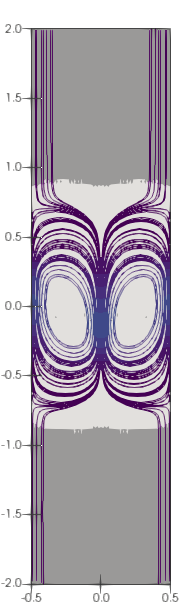}\includegraphics[height=0.55\textwidth]{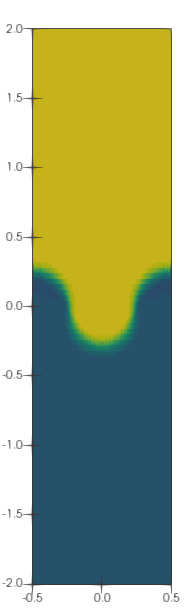}\hspace{5mm}\includegraphics[height=0.55\textwidth]{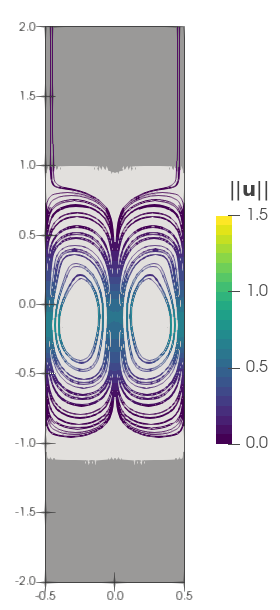}\includegraphics[height=0.55\textwidth]{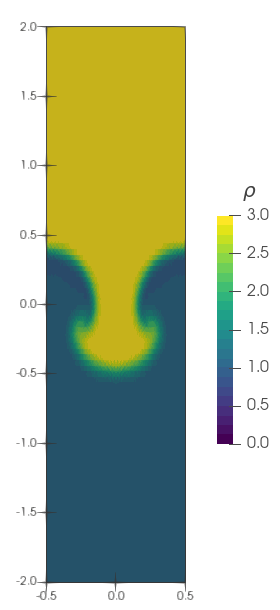}
	\end{center}
	\vspace{-3mm}
	\caption{Rayleigh-Taylor instability. Velocity stream lines and inactive set ($\mathcal{I}_\gamma$, dark gray); and density interface at times t = 1.0 and 1.5 for Re=1000 (top) and Re=3000 (bottom). Parameters: $\tau_s=0.1$, $\omega = 0.1$}\label{fig:ex1-3}  
\end{figure}

Now, we compare simulations employing two different Reynolds numbers: $\mathrm{Re} = 1000$ and $\mathrm{Re} = 3000$. The results displayed in Figure \ref{fig:ex1-3} show similar behavior to what was described for simulations of Newtonian fluids, in that viscosity plays no role in the velocity of the downward motion of the heavy fluid (see \cite{Calgaro2008,Freignaud2001}). We observe that as the Reynolds number increases, the velocity streamlines remain almost the same. However, note that rotating vortices are less developed in our simulation than what has been reported for the Newtonian counterpart, due to the influence of the plasticity threshold on the fluid dynamics. Additionally, the active zone (displayed in light gray) slightly decreases as the Reynolds number increases.

\subsection{Falling droplet}

 \begin{figure}[!t]
	\begin{center}
		\includegraphics[height=0.4\textwidth]{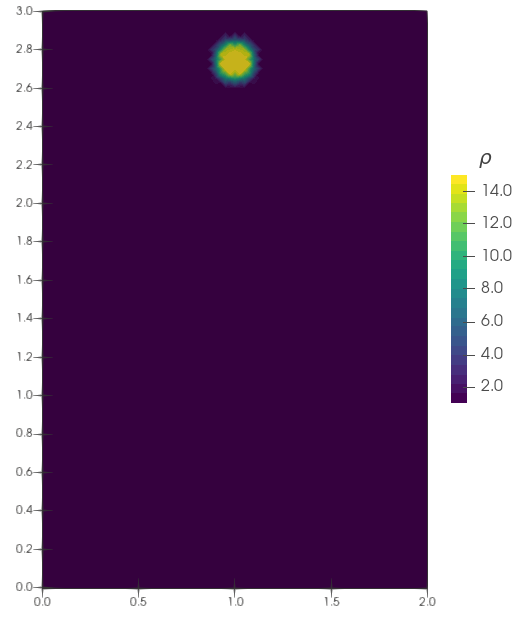}\includegraphics[height=0.4\textwidth]{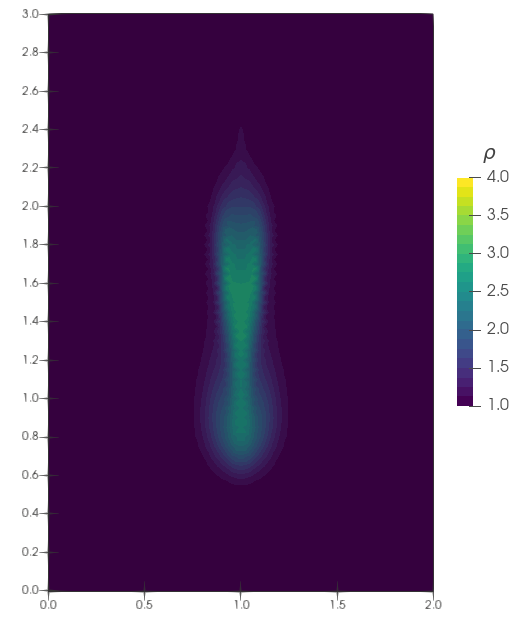}\includegraphics[height=0.4\textwidth]{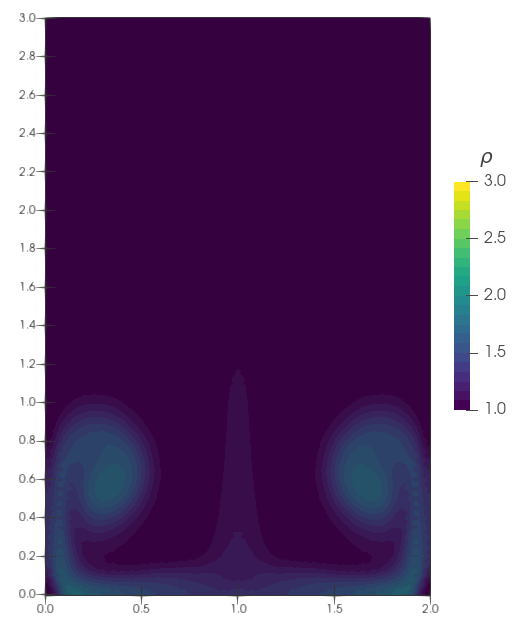} \\
		\includegraphics[height=0.4\textwidth]{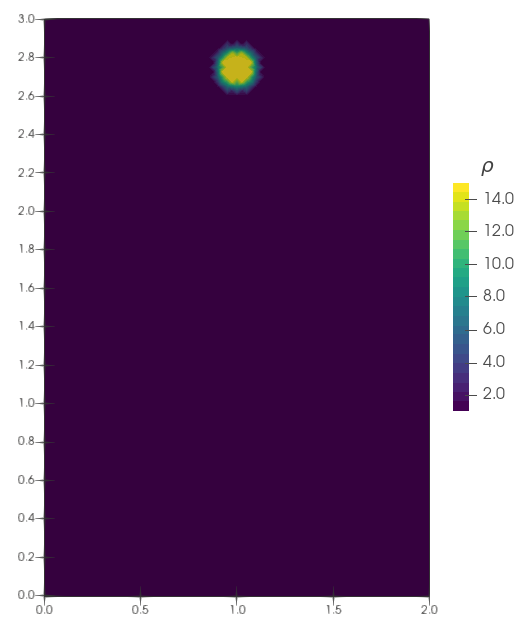}\includegraphics[height=0.4\textwidth]{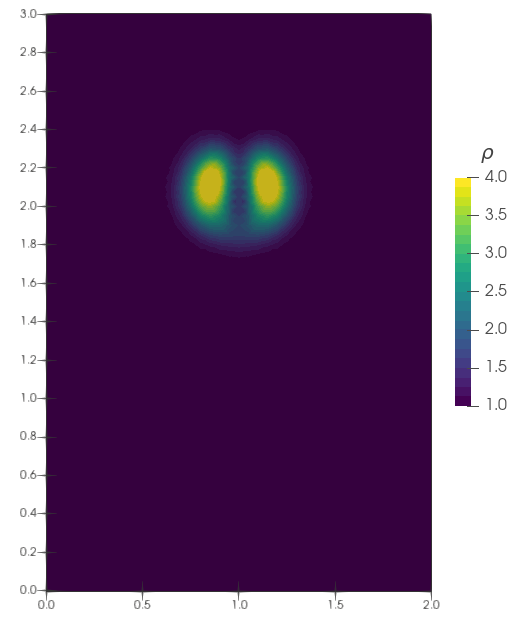}\includegraphics[height=0.4\textwidth]{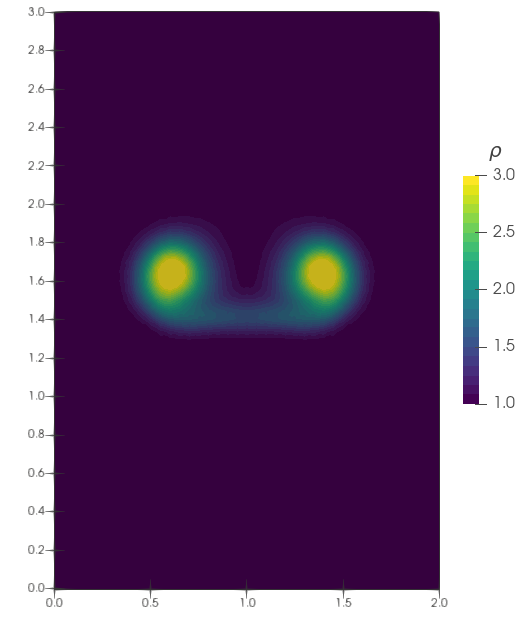}
	\end{center}
	\vspace{-3mm}
	\caption{Falling Droplet: Density interface at times t = 0.1, 2.0 and 4.0 for $\tau_s = 0.0 $ (top) and $\tau_s = 1.0$ (bottom). Parameters: $\mathrm{Re}=1000$}\label{fig:ex2-1}  
\end{figure}

Now we investigate a droplet falling through a light fluid. The domain is $\Omega = ]0,0[\times]l,1.5l[$, with $l=2$. At time $t=0$ the fluid is at rest with initial density given by
\begin{align*}
	\rho(x,y) = \begin{cases}
		15.0 &\text{if } 0\leq \sqrt{(x-1.0)^2+(y-2.75)^2}\leq 0.1 \\
		1.0  & \text{elsewhere}
	\end{cases}
\end{align*}

 \begin{figure}[!t]
	\begin{center}
		\includegraphics[height=0.42\textwidth]{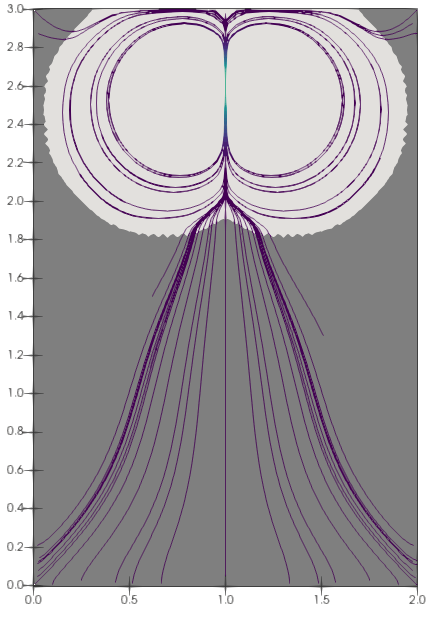}\includegraphics[height=0.42\textwidth]{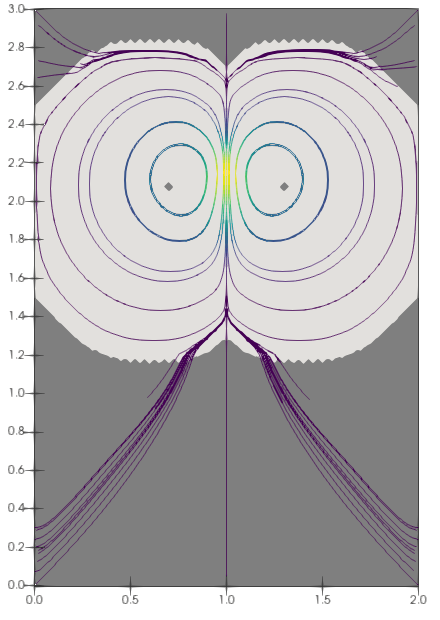}\includegraphics[height=0.42\textwidth]{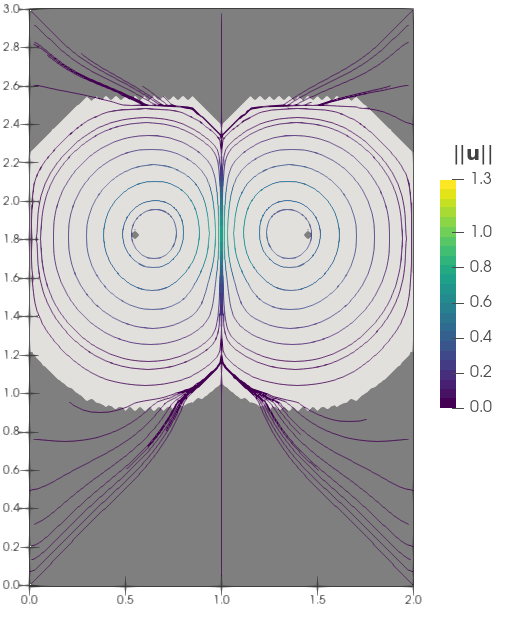}
	\end{center}
	\vspace{-3mm}
	\caption{Falling Droplet. Velocity vector field stream lines and inactive set ($\mathcal{I}_\gamma$, dark gray); at times t = 1.0, 2.0 and 3.0. Parameters: $\tau_s = 1.0 $, $\mathrm{Re}=1000$}\label{fig:ex2-2}  
\end{figure}

The equations are made dimensionless by using the same reference quantities as in the previous example. In our test, we use nonslip boundary conditions on all walls and a mesh of $40 \times 60$ cells. We set $\mathrm{Re} = 1000$ and test two cases: zero plasticity threshold ($\tau_s=0$) and $\tau_s = 1.0$. The qualitative results are displayed in Figures \ref{fig:ex2-1} and \ref{fig:ex2-2}. As we increase the value of $\tau_s$, the recirculation patterns around the downward droplet's path appear earlier, causing the droplet to split in half. While in the zero threshold case, the split occurs once the droplet reaches the domain bottom, in the $\tau_s=1.0$ case, the split starts as early as the $t=2.0$ snapshot. A closer look at the $\tau_s = 1.0$ case shows how the active zone moves following the droplet.

 \begin{figure}[!t]
	\begin{center}
		\includegraphics[height=0.40\textwidth]{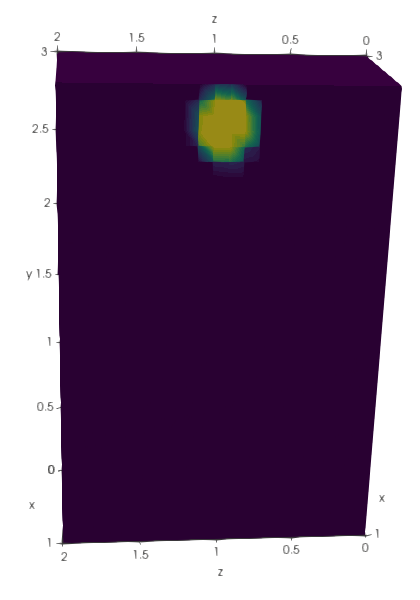}\includegraphics[height=0.40\textwidth]{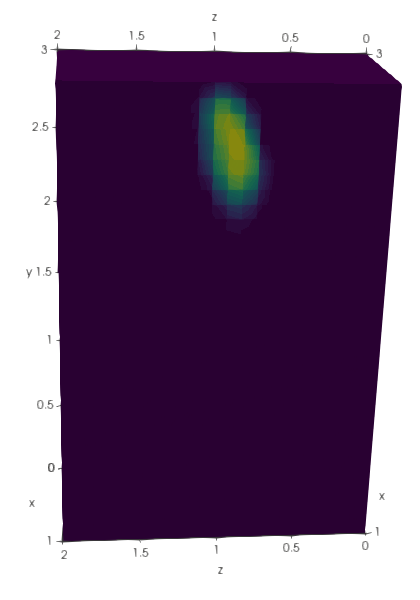}\includegraphics[height=0.40\textwidth]{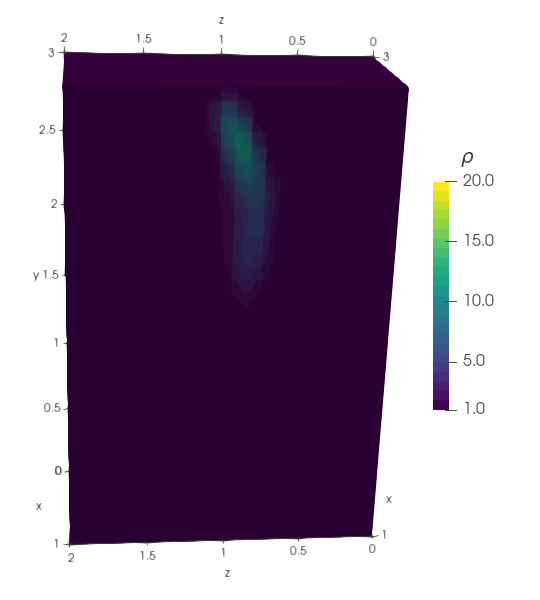}\\
		\includegraphics[height=0.37\textwidth]{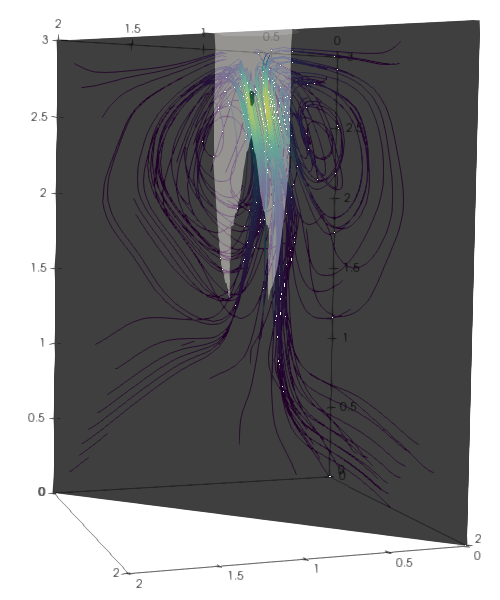}\includegraphics[height=0.37\textwidth]{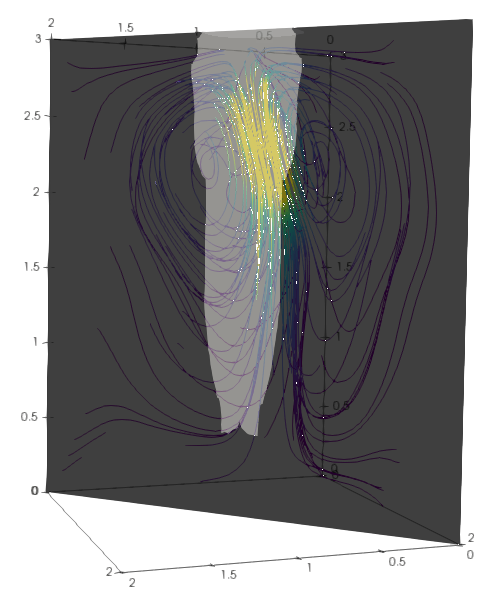}\includegraphics[height=0.37\textwidth]{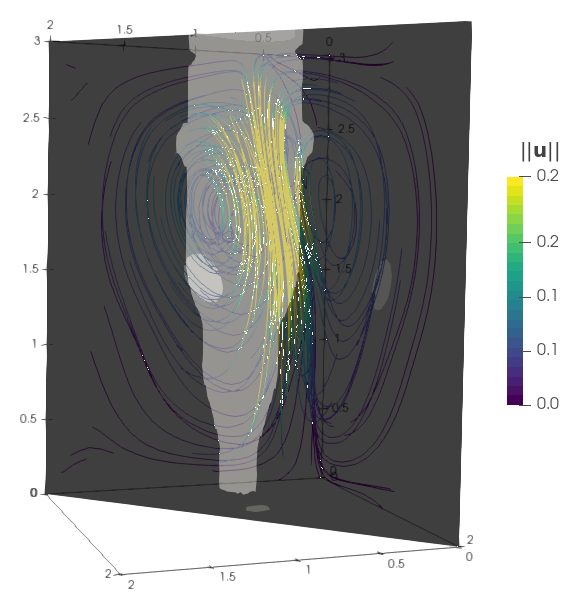}
	\end{center}
	\vspace{-3mm}
	\caption{3D Falling Droplet. Top: Density interface at times t=0.1, 1.0 and 2.0. Bottom: Velocity vector field stream lines and inactive set ($\mathcal{I}_\gamma$, dark gray); at times t = 0.1, 1.0 and 2.0. Parameters: $\tau_s = 0.5 $, $\mathrm{Re}=1000$}\label{fig:ex2-3}  
\end{figure}

Finally, a direct extension of this experiment was performed using a three-dimensional domain $\Omega = ]0,0,0[\times]l,1.5l,l[$ with $l=2$, and a tetrahedral grid with 95832 cells. We set 
\begin{align*}
	\rho(x,y) = \begin{cases}
		20.0 &\text{if } 0\leq \sqrt{(x-1.0)^2+(y-2.7)^2+(z-1.0)^2}\leq 0.2 \\
		1.0  & \text{elsewhere},
	\end{cases}
\end{align*}
$\mathrm{Re} = 1000$ and plasticity threshold ($\tau_s=0.5$). As shown in Figure \ref{fig:ex2-3}, it takes more time for the droplet to cover the same vertical distance in the three-dimensional domain. It is also noticeable from the bottom image in Figure \ref{fig:ex2-3} how the active zone (in light grey) is smaller in this three-dimensional case but grows as the droplet gains speed while following the path towards the bottom of the cell.

 \begin{figure}[!t]
	\begin{center}
		\includegraphics[height=0.4\textwidth]{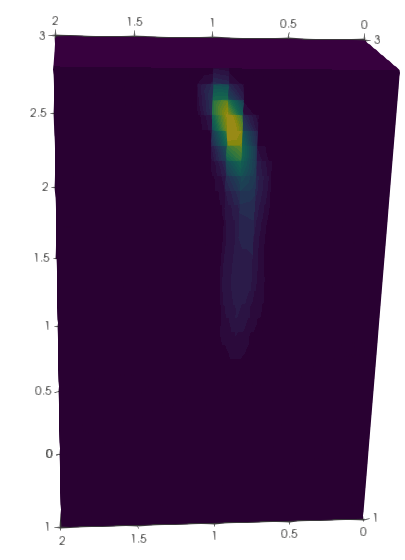}\includegraphics[height=0.4\textwidth]{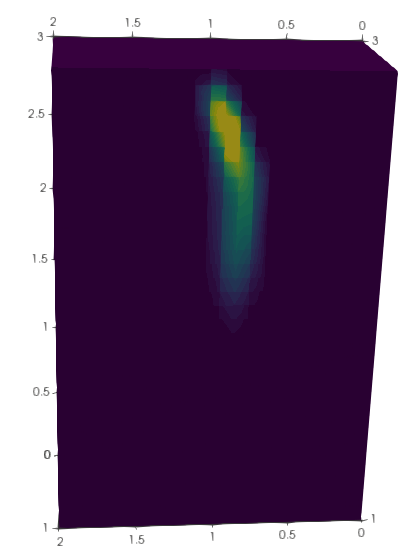}\includegraphics[height=0.4\textwidth]{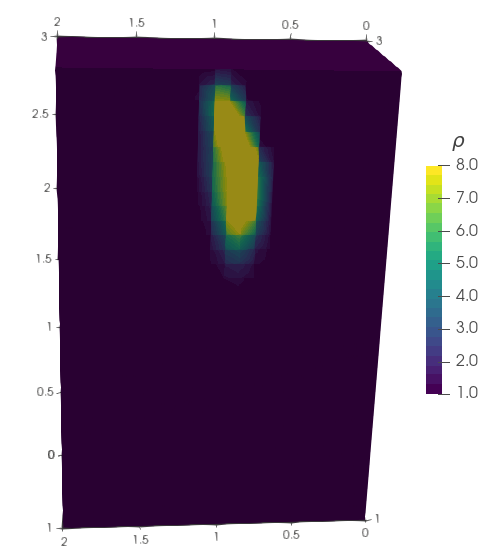}
	\end{center}
	\vspace{-3mm}
	\caption{3D Falling Droplet: Density interface at time t=2.5, for $\tau_s=$0, 2.0 and 5.0 (from left to right). Parameters: $\Delta_t =0.25$, Re=$1000$.}\label{fig:ex2-tau}  
\end{figure}

 \begin{figure}[!t]
	\begin{center}
		\includegraphics[height=0.5\textwidth]{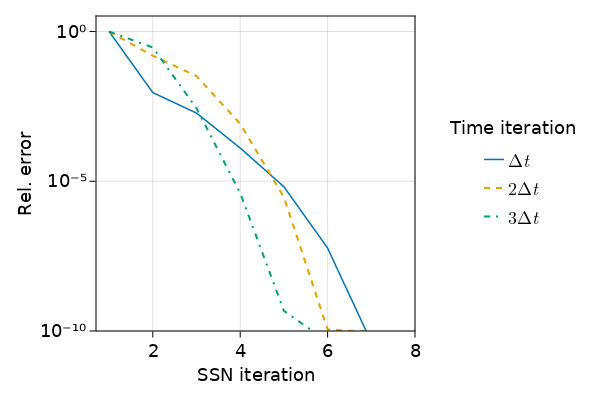}
	\end{center}
	\vspace{-3mm}
	\caption{3D Falling Droplet: Relative Error vs SSN iterations for the first three time iterations. Parameters: $\tau_s=0.5$, $\Delta_t =0.2$, Re=$1000$.}\label{fig:ex2-conv}  
\end{figure}

Relative error for each SSN iteration is displayed in Figure \ref{fig:ex2-conv}. As in the previous tests, convergence is slower for the first-time iterations (including the backward Euler step), and then it becomes faster as the initial approximation for the SSN iterations improves.

\section{Conclusions} \label{sec:conclusions}

In this work, we present a second-order divergence-conforming dG method for the case of Huber-regularized Bingham flows with variable density. We introduce the Huber regularization and show its qualitative advantages when used in this kind of model. The numerical scheme is based on a discontinuous Galerkin formulation for the mass density equation, stabilized with an upwind term, coupled with a divergence-conforming approximation of a Huber-type regularization of the Bingham flow equation, and uses a BDF2 scheme for the time integration of the mass conservation and momentum equations. In each time step, we solve the resulting system of the space discretization using a Semismooth Newton Iteration, which is suitable due to the Huber regularization step. We prove the stability of the continuous problem and the stability of the full-discrete scheme. To verify the correctness of the method, we compare our qualitative results with test cases previously considered in the literature. For instance, when simulating the evolution of the Rayleigh-Taylor instability of the interface between fluids of different densities, the results of the method with a low yield stress threshold agree with the variable density Navier-Stokes computations in \cite{Calgaro2008, Freignaud2001}, especially in the early stages of vortex formation and roll-up. Furthermore, simulations with different mesh refinements show that we can still capture the main features of the density front even with rather coarse meshes, while finer details improve with mesh refinement.
	The spatial convergence analysis conducted for the stationary Bingham test problem indicates that the $H(\mathrm{div})$-conforming method is accurate enough to consider the computed solutions for the homogeneous Bingham case as a reliable base for model extensions, such as the variable density case studied here. We leave open for future studies the more complex case where the variation in density also affects the rheological model, for instance, through a change in the yield stress. Nevertheless, we consider our results to support the general conclusion that the scheme is worthy of attention for the numerical approximation of complex fluids with yield.

\bigskip
\noindent\textbf{Acknowledgement.} We acknowledge the partial support by Escuela Polit\'ecnica Nacional del Ecuador, under the projects PIS 18-03 and PIGR 19-02. We are also grateful with the anonymous reviewers whose comments helped us to improve the article. This research was carried out by using the research computing facilities offered by the Scientific Computing Laboratory of the Research Center on Mathematical Modeling: MODEMAT, Escuela Polit\'ecnica Nacional - Quito.

\end{document}